\documentclass[12pt,reqno]{amsart}
\usepackage{amsfonts,amssymb,amsmath}
\usepackage{datetime}
\usepackage{xcolor}
\usepackage{enumerate}
\usepackage{multirow}
\usepackage{float}
\makeatletter\@namedef{subjclassname@2020}{\textup{2020} Mathematics Subject Classification}
\makeatother

\definecolor{fgreen}{RGB}{44,144, 14}

\renewenvironment{proof}{{\bfseries Proof.}}{\qed}

\numberwithin{equation}{section} 
\newtheorem{theorem}{Theorem}[section] 
\newtheorem{proposition}[theorem]{Proposition} 
 
\newtheorem{lemma}[theorem]{Lemma} 

\theoremstyle{definition}
\newtheorem{definition}[theorem]{Definition} 
\newtheorem{remark}[theorem]{Remark}

\usepackage{cite}
\usepackage[colorlinks=true,citecolor=cyan,urlcolor=blue,linkcolor=blue]{hyperref}

\def\R{\mathbb R}

\def\C{\mathbb C}

\def\H{\mathbb H}

\def\z{{\bf z}}

\def\N{\mathbb N}
\def\Z{\mathbb Z}

\def\R{\mathbb R}
\def\e{\mathcal E}

\def\P{\mathbb P}

\def \y  {{\bf y}}

\def\R{\mathbb {R}}
\def\C{\mathbb {C}}
\def\N{\mathbb {N}}
\def\H{\mathbb {H}}
\def\Z{\mathbb {Z}}

\def\z{\mathfrak {z}}
\def\c{\mathfrak {c}}

\def\b{\mathfrak b}

\newcommand{\defref}[1]{Definition~\ref{#1}}
\newcommand{\secref}[1]{Section~\ref{#1}}

\newcommand{\lemref}[1]{Lemma~\ref{#1}}
\newcommand{\remref}[1]{Remark~\ref{#1}}

\begin{document} 
	
\title[Limit sets of cyclic quaternionic Kleinian groups]{Limit sets of cyclic quaternionic Kleinian groups}
\author[S.  Dutta, K. Gongopadhyay and T. Lohan]{Sandipan Dutta, Krishnendu Gongopadhyay and 
	Tejbir Lohan}

\address{Indian Institute of Science Education and Research (IISER) Mohali,
	Knowledge City,  Sector 81, S.A.S. Nagar 140306, Punjab, India}
\email{sandipandutta98@gmail.com }

\address{Indian Institute of Science Education and Research (IISER) Mohali,
	Knowledge City,  Sector 81, S.A.S. Nagar 140306, Punjab, India}
\email{krishnendug@gmail.com, krishnendu@iisermohali.ac.in}

\address{Indian Institute of Science Education and Research (IISER) Mohali,
	Knowledge City,  Sector 81, S.A.S. Nagar 140306, Punjab, India}
\email{tejbirlohan70@gmail.com}



\begin{abstract}
In this paper, we consider the natural action of ${\rm SL}(3, \H)$ on the quaternionic projective space $ \P_{\H}^2$. Under this action, we investigate limit sets for cyclic subgroups of ${\rm SL}(3, \H)$. We compute two types of limit sets, which were introduced by Kulkarni and Conze-Guivarc'h, respectively.
\end{abstract}

\maketitle 
\section{Introduction} 
\label{sec:intro}
Classically, the Kleinian groups arise as discrete subgroups of ${\rm SL}(2, \C)$ acting on the Riemann sphere $ \P_{\C}^1$, which is the boundary of the three-dimensional hyperbolic space. The action is in terms of the Möbius transformations. The Kleinian groups played a significant role in several areas of mathematics due to their algebraic, geometric, and dynamical significance, see \cite{ma1}, \cite{ma2}. The limit sets are the starting points toward understanding the Kleinian group’s dynamics. A Kleinian group acts properly discontinuously on the complement of its limit set. In the classical setup, the structure of the limit sets analogously related to the holomorphic dynamics of iterated functions of the Riemann sphere via the Sullivan dictionary, cf. \cite{mc}.  
In an attempt to generalise the Sullivan dictionary in the higher dimension, Seade and Verjovsky \cite{SV01,SV02,SV03} initiated the investigation of limit sets of discrete subgroups of ${\rm SL}(3, \C)$ that act on the two-dimensional complex projective space $\P^2_{\C}$. The theory of complex Kleinian group is also closely related to the understanding of discrete subgroups acting on the two-dimensional complex hyperbolic space because of the embedding of ${\rm SU}(2,1)$ inside ${\rm SL}(3, \C)$. A complex Kleinian group means a discrete subgroup of ${\rm PSL}(3, \C)$ with a non-empty domain of discontinuity. It is a problem of broad interest to obtain a maximal domain of discontinuity for a complex Kleinian group. On the other hand,  in \cite{Ku}, Kulkarni developed the notion of a generalised limit set for any group acting on a  locally compact Hausdorff topological space. Cano et al. have systematically investigated the Kulkarni limit sets and the Kulkarni domain of discontinuity for a wide class of complex Kleinian groups, see  \cite{cns}, \cite{can2}, \cite{can},\cite{Na}. The Kulkarni limit set may not give the maximal domain of discontinuity, but it certainly gives a domain of discontinuity for the group under consideration. This may help in understanding the  dynamics of the group. 

The aim of this paper is to initiate an investigation of "quaternionic Kleinian groups", that is, discrete subgroups of ${\rm PSL}(3, \H)$ with a non-empty domain of discontinuity. A starting point to do that is to understand the limit sets of the cyclic subgroups of ${\rm PSL}(3, \H)$. In this paper, we compute the Kulkarni limit sets of cyclic subgroups of ${\rm PSL}(3, \H)$. We have worked following the spirit of Navarette \cite{TR} who computed the Kulkarni limit sets of cyclic subgroups of ${\rm PSL}(3, \C)$. 

We shall describe the limit sets in three tables. Before that, in the following, we recall the notion of the Kulkarni limit set.  

\begin{definition}
	Let $P_X=\{A_\beta \mid  \beta \in B\}$ be a family of subsets of $X$, where $B$ is an infinite indexing set. A point $p$  in $X$ is called a \textit{cluster point} of $P_X$ if every neighbourhood
	of $p$ intersects $A_\beta$ for infinitely many  $\beta$ in $B$.
\end{definition}

Consider the natural action of a subgroup  $G$  of ${\rm SL}(3, \H)$ on the two-dimensional quaternionic projective space $X= \P_{\H}^2$. The isotropy subgroup of any $x\in X$ is defined as $G_x=\{g\in G \mid  gx=x\}$. Consider the following three sets
\begin{enumerate}[(a)]
	\item $L_0(G):=$ the closure of the set of points of $X$ which have an infinite isotropy group,
	\item $L_1(G):=$ the closure of the cluster points of orbits of points in $X\setminus L_0(G)$, and
	\item $L_2(G):=$ the closure of the cluster points of $\{g(K)\}_{g\in G}$, where $K$ runs over all the compact subsets of $X\setminus \{L_0(G)\cup L_1(G)\}$.
\end{enumerate}

\begin{definition}\label{def-kul-limitset}
	With $G$ and $X$ as defined above, the \textit{Kulkarni limit set} of $G$ is
	\begin{equation}
		\Lambda(G):=L_0(G)\cup L_1(G)\cup L_2(G).
	\end{equation}
	The \textit{Kulkarni domain of discontinuity} of $G$ is defined as $\Omega(G)=X\setminus \Lambda(G)$.
\end{definition}

We remark that the \defref{def-kul-limitset} is valid for any group $G$ acting on a locally compact Hausdorff space $X$ and 
Kulkarni proved the following proposition in \cite{Ku}.
\begin{proposition}[cf.~{\cite[ Proposition 1.3]{Ku}}] 
	Let $X$ be a locally compact Hausdorff space and $G$ be a group acting on $X$, then $G$ acts properly discontinuously on $\Omega(G)$. Also, $\Omega(G)$ is a $G$-invariant open subset of $X$.  Further,  if $\Omega(G)\neq \phi$,  then $G$ is discrete.
\end{proposition}

\begin{definition}
	A discrete subgroup $G\subset \mathrm{PSL}(3,\mathbb{H})$ is called \textit{quaternionic Kleinian} if there exists a non-empty open invariant set where the action is properly discontinuous, i.e., $\Omega(G)\neq \phi$.
\end{definition}
{{Since ${\rm SL}(3, \H)$ is a double cover of ${\rm PSL}(3, \H)$ by $\{\pm \mathrm{ I}_3\}$}}, we often lift an element from the projective group ${\rm PSL}(3, \H)$ to ${\rm SL}(3, \H)$ and will consider its matrix representation. 
To classify the limit sets, we pick up a Jordan form in ${\rm SL}(3, \H)$ and compute the limit set for the cyclic subgroup generated by that Jordan form. Note that each eigenvalue class of a quaternionic matrix contains a unique complex representative with non-negative imaginary part. This gives us the following classification of elements of ${\rm SL}(3, \H)$ into three mutually exclusive classes.

\begin{definition}	\label{def:classification}
	Let $g$ be an element in ${\rm SL(3, \H)}$. 
	\begin{enumerate}[(i)]
		\item $g$ is called elliptic if it is semisimple and the eigenvalue classes are represented by unit modulus complex numbers. In other words, the complex representatives of the eigenvalues are $e^{i \alpha}$, $e^{i \beta}$,  $e^{i \gamma}$, where $\alpha, \beta, \gamma \in [0,\pi]$.
		
		\item $g$ is called loxodromic if not all the eigenvalue classes are represented by unit modulus complex numbers.  
		
		\item $g$ is called parabolic if it is not semisimple, and the eigenvalue classes are represented by unit modulus complex numbers. 
	\end{enumerate}
\end{definition} 

The above classes can be divided further into several subclasses: an elliptic element is rational elliptic if all of $\alpha$, $\beta$, and $\gamma$ are rational. A loxodromic element is regular loxodromic if all its eigenvalues have distinct moduli. If all the eigenvalues of a parabolic element are $1$, we call it unipotent. Otherwise, we call it ellipto-parabolic or ellipto-translation according as the corresponding minimal polynomial has a  factor of degree $2$ or $3$. We note here that the nomenclature for the subclasses used in this paper is influenced by similar names from Navarette's work \cite{TR}. We refer to the subsequent sections for details of the subclasses. 

Let $\{e_1, e_2, e_3\} \subset  \P_{\H}^2$ be the projection of the standard basis $ \{ \textbf{e}_1, \textbf{e}_2, \textbf{e}_3 \}$ of $\H^3$ to the projective space $\P_{\H}^2$. 
We refer to \secref{subsec-quater-line} for the notion of quaternionic projective space, and
$\mathbb{L}\{p,q\}$ denotes the quaternionic projective line passing through the distinct points $p, q  \in \P_{\H}^2$. We have summarized the Kulkarni limit sets for the cyclic subgroups of ${\rm PSL}(3, \H)$ in Table \ref{table:1}, Table \ref{table:2}, and Table \ref{table:3}.

\begin{table}[H]	
	\centering{
		
		\caption{Kulkarni sets for cyclic subgroups of elliptic elements of $\mathrm{PSL}(3,\mathbb{H})$}
		\begin{tabular}{|p{6cm}|p{4cm}|p{1.5cm}|p{1.5cm}|p{1.5cm}|}
			\hline
			\textbf{Kulkarni Set} & $\mathbf{L_0(G)}$  & $\mathbf{L_1(G)}$ & $\mathbf{L_2(G)}$ & $\mathbf{\Lambda(G)}$\\
			\hline
			\textit{Rational elliptic} & $\phi$ &  $\phi$ &  $\phi$ &  $\phi$ \\
			\hline
			\textit{Simple irrational elliptic} & $\mathbb{P}^2_\mathbb{C}$  & $\mathbb{P}^2_\mathbb{H}$  & $\phi$ & $\mathbb{P}^2_\mathbb{H}$\\
		\hline
			\textit{Compound irrational elliptic} & set of fixed points & $\mathbb{P}^2_\mathbb{H}$ & $\phi$ & $\mathbb{P}^2_\mathbb{H}$\\
			\hline
		\end{tabular}
		\label{table:1}}
\end{table}

\begin{table}[H]
	
	\centering{
		
		\caption{Kulkarni sets for cyclic subgroups of loxodromic elements of $\mathrm{PSL}(3,\mathbb{H})$}
		\begin{tabular}{|p{2.4cm}|p{2.8cm}|p{2.8cm}|p{3.5cm}|p{3.5cm}|}
			
			\hline
			\textbf{Kulkarni Set} & $\mathbf{L_0(G)}$  & $\mathbf{L_1(G)}$ & $\mathbf{L_2(G)}$ & $\mathbf{\Lambda(G)}$\\
			\hline
			\textit{Regular $loxodromic$} & $\{e_1,e_2,e_3\}$ & $\{e_1,e_2,e_3\}$ & $\mathbb{L}{ \{e_1,e_2\}}\cup \mathbb{L}{ \{e_2,e_3\}} $ & $\mathbb{L}{ \{e_1,e_2\}}\cup\mathbb{L}{\{e_2,e_3\}} $ \\
			\hline
			\textit{Screw} & $\{e_1,e_2,e_3\}$ &$\mathbb{L} {\{e_1,e_2\}}\cup\{e_3\}$ & $\mathbb{L}{ \{e_1,e_2\}}\cup\{e_3\}$ & $\mathbb{L}{ \{e_1,e_2\}}\cup\{e_3\}$ \\
			\hline
			\textit{Homothety} & set of fixed points & $\mathbb{L}{ \{e_1,e_2\}}\cup\{e_3\}$ & $\mathbb{L}{ \{e_1,e_2\}}\cup\{e_3\}$ & $\mathbb{L}{ \{e_1,e_2\}}\cup\{e_3\}$ \\
			\hline
			\textit{Loxo-parabolic} & $\{e_1,e_3\}$ & $\{e_1,e_3\}$ & $\mathbb{L}{\{e_1,e_2\}}\cup \mathbb{L} {\{e_1,e_3\}} $ & $\mathbb{L}{\{e_1,e_2\}}\cup \mathbb{L}{ \{e_1,e_3\} }$ \\
			\hline
		\end{tabular}
		\label{table:2}}
\end{table}

\begin{table}[H]
	
	\centering{
		
		\caption{Kulkarni sets for cyclic subgroups of parabolic elements of $\mathrm{PSL}(3,\mathbb{H})$}
		\begin{tabular}{|p{6cm}|p{2cm}|p{2cm}|p{2cm}|p{2cm}|}
			
		\hline
	
			\textbf{Kulkarni Set} & $\mathbf{L_0(G)}$  & $\mathbf{L_1(G)}$ & $\mathbf{L_2(G)}$ & $\mathbf{\Lambda(G)}$\\
			\hline
			\textit{Vertical  translation} & $\mathbb{L}\{e_1,e_3\}$ & $\{e_1\}$ & $\{e_1\}$ & $\mathbb{L}\{e_1,e_3\}$ \\
			\hline
			\textit{Non-vertical translation} & $\{e_1\}$ & ${\{e_1\}}$ & $\mathbb{L}{ \{e_1,e_2\}}$ & $\mathbb{L}{ \{e_1,e_2\}}$ \\
			\hline
			\textit{Rational ellipto-parabolic} & $\mathbb{L}{ \{e_1,e_3\}}$ & ${ \{e_1\}}$ &$ \{e_1\}$ & $\mathbb{L}{ \{e_1,e_3\}}$ \\
			\hline
			\textit{Irrational ellipto-parabolic} & ${ \{e_1,e_3\}}$ & $\mathbb{L}{ \{e_1,e_3\}}$ & $\{e_1\}$  & $\mathbb{L}{ \{e_1,e_3\}}$ \\
			\hline
			\textit{Ellipto-translation} & $\{e_1\}$ & $\{e_1\}$ & $\mathbb{L}\{e_1,e_2\}$ & $\mathbb{L}\{e_1,e_2\}$\\
		\hline
		\end{tabular}
		\label{table:3}}
\end{table}

There is another notion of a limit set for the action of a discrete subgroup on a linear space that was introduced by Conze and Guivarc’h, cf.  \cite{Conze}. The  notion of Conze-Guivarc’h limit set is inspired by the ideas of proximal transformations introduced by Abels, Margulis, and Soĭfer in \cite{AMS}.  

\begin{definition}\label{proximal}
	An element $g$ in $\mathrm{SL}(3,\mathbb{H})$ is called proximal if it has a maximal norm eigenvalue. In the case of $\mathrm{PSL}(3,\mathbb{H})$ we say that an element is proximal if its lift is proximal.
\end{definition}

\begin{definition}\label{conze}
	Let $G \subset \mathrm{PSL}(3,\mathbb{H})$ that contains proximal elements and whose action on $ \P_{\H}^2$ is strongly irreducible, i.e., there does not exist any proper non-zero subspace of $ \P_{\H}^2$  invariant under the action of a subgroup of finite index in $G$.  The \textit{Conze-Guivarc’h limit} set is the closure of the subset of
	$\mathbb{P}^2_\mathbb{H}$ consisting of all the attracting fixed points of proximal elements of $G$.
\end{definition}

Note that the Conze-Guivarc’h limit set is only applicable to those discrete subgroups of $\mathrm{PSL}(3,\mathbb{H})$ with proximal elements,  but not every discrete subgroup contains proximal elements. In \cite{DU},  Barrera et al. considered the action of discrete subgroups of $ \mathrm{PSL}(3,\mathbb{C})$  on the dual space of $\mathbb{P}^2_\mathbb{C}$  and introduced the notion of  {\it extended Conze-Guivarc’h limit set} for complex Kleinian groups. These have been further studied in  \cite{PS}; also see \cite{cns} and the references therein. We now extend the definition from \cite{DU} to the quaternionic setup. 
\begin{definition} \label{def:conze2}
	Let us consider $G\subset \mathrm{PSL}(3,\mathbb{H})$, acting on the dual space $(\mathbb{P}^2_\mathbb{H})^*$ which is the space of lines in $\mathbb{P}^2_\mathbb{H}$.  We say that $q\in (\mathbb{P}^2_\mathbb{H})^*$ is a limit point of $G$ if there exists an open subset $U\subset (\mathbb{P}^2_\mathbb{H})^*$ and there exists a sequence $\{g_n\}\subset G$ of distinct elements such that for every $p\in U$,  $\lim_{n\rightarrow \infty} g_n\cdot p=q.$
	The set of limit points will be called the \textit{extended Conze-Guivarc’h limit set}, denoted by $\hat{L}(G)$.
\end{definition}
In Section \ref{sec:conze}, we have discussed the extended Conze-Guivarc’h limit set for cyclic subgroups of $\mathrm{PSL}(3,\mathbb{H})$. Table \ref{table:4} gives the extended Conze-Guivarc’h limit sets for cyclic subgroups of ${\rm PSL}(3, \H)$.

\begin{table}[H]
	
	\centering{
		
		\caption{Extended Conze-Guivarc’h limit sets for cyclic subgroups of $\mathrm{PSL}(3,\mathbb{H})$}
		\begin{tabular}{||p{1.8cm}|p{.9cm}||p{3.3cm}|p{1.5cm}||p{2.6cm}|p{3.5cm}||}
			
			\hline\hline
			\textbf{Elliptic} & $\mathbf{\hat{L}(G)}$  & \textbf{Parabolic} & $\mathbf{\hat{L}(G)}$ & \textbf{Loxodromic} &  $\mathbf{\hat{L}(G)}$\\
			\hline
			\textit{Rational elliptic} & $\phi$ & \textit{Vertical translation}  & $\mathbb{L}{ \{e_1,e_3\}}$ & \textit{Regular $loxodromic$} & $\mathbb{L}\{e_1,e_2\}\cup\mathbb{L}\{e_2,e_3\}$\\
			\hline
			\textit{Simple irrational elliptic} & $(\mathbb{P}^2_\mathbb{H})^*$ & \textit{Non-vertical translation} & $\mathbb{L}{ \{e_1,e_2\}}$ & \textit{Screw} & $\mathbb{L}{ \{e_1,e_2\}}$\\
			\hline
			\textit{Compound irrational elliptic} & $(\mathbb{P}^2_\mathbb{H})^*$ &  \textit{Ellipto-parabolic} & $\mathbb{L}{\{e_1,e_3\}}$ & \textit{Homothety} & $\mathbb{L}{ \{e_1,e_2\}}$ \\
			\hline
			&  &\textit{Ellipto-translation}& $\mathbb{L}{ \{e_1,e_2\}}$   & \textit{Loxo-parabolic} & $\mathbb{L}{ \{e_1,e_2\}} \cup \mathbb{L}\{e_1,e_3\}$\\
			\hline\hline
		\end{tabular}
		\label{table:4}}
	\end{table}

Note that even though the Jordan forms in $\mathrm{SL}(3,\mathbb{H})$ (see \lemref{lem-Jordan-M(n,H)}) are given by complex matrices, the computation of the limit sets is not a straightforward adaption of the complex case. There are also a few cases that do not appear as complex Kleinian groups. We can consider the lift of a Jordan matrix representative of a given quaternionic projective transformation only up to scaling by a real number in the quaternionic case. Unlike the complex case, here we can not take lift $\mathrm{diag}(1,1,e^{2\pi i (\beta - \alpha)})$ of a quaternionic projective transformation represented by the Jordan form $\mathrm{diag}(e^{2\pi i\alpha},e^{2\pi i \alpha},e^{2\pi i \beta})$, where $\alpha, \beta \in (0,\pi)$ (cf. \cite{cns}).  Due to the non-commutativity of the quaternions, this particular property not only differentiates from the corresponding results in the theory of the complex Kleinian groups but also gives us new cases to consider and investigate. For example, the quaternionic projective transformation represented by the Jordan form $\mathrm{diag}(e^{2\pi i\alpha},e^{2\pi i \beta},e^{2\pi i \gamma}),\; \alpha,\beta,\gamma\in \mathbb{R}\setminus\mathbb{Q}$, is completely a new type of projective transformation that appears in the quaternionic case. 

Finally, we remark that in the complex case, Cano et al.  \cite{clu} have investigated the dynamical classification of elements and the Kulkarni limit sets for cyclic subgroups of ${\rm PSL}(n+1, \C)$. It is expected that it should be possible to generalise their work to ${\rm PSL}(n+1, \H)$ as well. 

\subsection{Structure of the paper}
We discuss the preliminaries in the \secref{prelim}. The third, fourth, and fifth sections are devoted to proving our results about the Kulkarni limit sets for cyclic subgroups of ${\rm PSL}(3, \H)$ generated by elliptic,  loxodromic and parabolic elements, respectively.   In the \secref{sec:conze},  we have studied the extended Conze-Guivarc’h limit sets for cyclic subgroups of ${\rm PSL}(3, \H)$ and proved Table \ref{table:4}.

\section{Preliminaries}\label{prelim}
Let $\H$ denote the division ring of Hamilton's quaternions.  We recall that every element in $\H$ can be expressed as $a=a_0 + a_1 i + a_2 j + a_3 k$, where $i^2=j^2=k^2= i j k =-1$,  and $a_0, a_1, a_2, a_3 \in \R$.  The conjugate of $a$ is given by $\bar a=a_0-a_1 i -a_2 j -a_3k$.   We identify the real subspace $\R \oplus \R i$ with the usual complex plane $\C$.  For an elaborate discussion on the theory of matrices over the quaternions,  see \cite{rodman}, \cite{Zh}.

\begin{definition}\label{def-eigen-M(n,H)} Let $A$ be an arbitrary element of  the algebra $ \mathrm{M}(n,\H)$ of all $n \times n$ matrices over $\H$.
A non-zero vector $v \in \H^n $ is said to be a (right) eigenvector of $A$ corresponding to a (right) eigenvalue $\lambda \in \H $ if the equality $ Av = v\lambda $ holds.
\end{definition}

Eigenvalues of $A$ occur in similarity classes, i.e., if $v$ is a eigenvector corresponding to the eigenvalue $\lambda$, then $v \mu \in v \mathbb H$ is a eigenvector corresponding to the eigenvalue $\mu^{-1} \lambda \mu$.  Each similarity class of eigenvalues contains a unique complex number with  non-negative imaginary part.  Here we shall represent each similarity class of eigenvalues by the \textit{unique complex representative} with non-negative imaginary part and refer to them as \textit{eigenvalues}.

Let $A \in  \mathrm{M}(n,\H)$ and write $ A = A_1 + A_2 \, j $,  where $ A_1,  A_2 \in \mathrm{M}(n,\C)$.  Now consider the embedding $ \Phi:  \mathrm{M}(n,\H)  \longrightarrow  \mathrm{M}(2n,\C)$ defined as 
\begin{equation*}\label{eq-embed-phi}
	\Phi(A) = \begin{pmatrix} A_1   &  A_2 \\
		- \overline{A_2} & \overline{A_1}  
	\end{pmatrix}, 
\end{equation*}
where $\overline{A_i}$ denotes the complex conjugate of  $A_i$. 

\begin{definition}
For $A \in  \mathrm{M}(n,\H)$, the determinant of $A$ is denoted by $ {\rm det}_{\H}(A)$  and is defined as  the determinant of  the corresponding matrix $ \Phi(A)$, i.e., $ {\rm det}_{\H}(A):=  {\rm det}(\Phi(A))$,  see   \cite[\S 5.9]{rodman}.
\end{definition}

Consider the Lie groups  $\mathrm{GL}(n,\H) = \{ g \in  \mathrm{M}(n,\H) \mid {\rm det_{\H}}(g)  \neq 0 \}$ and $  \mathrm{SL}(n,\H) = \{ g \in \mathrm{GL}(n,\H) \mid {\rm det_{\H}}(g)  = 1 \}.$

\begin{definition} [cf.~{\cite[p.  94]{rodman}}] \label{def-jordan}
A Jordan block $ \mathrm{J}(\lambda,m)$ is a $m \times m$ matrix with $ \lambda $ on the diagonal entries,  1 on all of the super-diagonal entries, and zero elsewhere.  We will refer to a block diagonal matrix where each block is a Jordan block as  \textit{Jordan form}. 
\end{definition}

Next we recall the Jordan  decomposition in $\mathrm{M}(n,\H)$,  see {\cite[Theorem 5.5.3]{rodman}}.
\begin{lemma}[{Jordan  forms in $ \mathrm{M}(n,\H)$, cf.~	 \cite{rodman}}] \label{lem-Jordan-M(n,H)}
For every $A \in  \mathrm{M}(n,\H)$ there is a invertible matrix $S \in  \mathrm{GL}(n,\H)$ such that $SAS^{-1}$ has the form 
	\begin{equation} \label{equ-Jordan-M(n,H)}
		SAS^{-1} =  \mathrm{J}(\lambda_1, m_1) \oplus  \cdots \oplus  \mathrm{J}(\lambda_k, m_k)
	\end{equation}
	where $ \lambda_1,  \dots,  \lambda_k $ are complex numbers (not necessarily distinct) and have non-negative imaginary parts. The form (\ref{equ-Jordan-M(n,H)}) is uniquely determined by $A$ up to a permutation of Jordan blocks.
\end{lemma} 

\subsection{The Quaternionic Projective Space $\P^2_\H$} \label{subsec-quater-line}
Consider the (right)  quaternionic space $\H^{3}$.  The two-dimensional quaternionic projective space $\P^2_\H$  is formed by taking the quotient of $\H^{3}\setminus \{0\}$ under the equivalence relation $ \sim $ in $\H^{3}\setminus \{0\}$,  where  $ z \sim w \Leftrightarrow z = w \, \alpha$  for some non-zero quaternion $\alpha$.  Let  $ \P : \H^{3} \setminus \{0\} \longrightarrow  \P^2_\H$  be the corresponding  quotient map.  A non-empty set  $W \subseteq \P^2_\H$ is said to be a projective subspace of dimension $k$ if there exists a   $(k + 1)$-dimensional $\H$-linear subspace $\widetilde{W}$ of  $\H^{3}$  such that $\P (\widetilde{W} \setminus  \{0\} ) = W$.  The quaternionic projective subspaces of dimension $1$ are called \textit{lines}.  Let $p = (x,y,z) \in \H^3 $ then we will use notation  $ [x:y:z]$ to denote $\P(p)$.  Note that $[x:y:z] = [x \alpha: y\alpha: z \alpha]$ for all $ (x,y,z) \in \H^3$ and non-zero quaternion $ \alpha \in \H$.
Given a set of points ${S}$ in $\P^2_\H$, we define
$$\langle{\,S\,}\rangle = \bigcap \{W \subseteq \P^2_\H : W  \hbox { is a projective subspace containing } {S} \}.$$
Clearly, $\langle{\,S\,}\rangle$ is a projective subspace of $\P^2_\H$.  If $p,q$ are distinct points of $\P^2_\H$ then $\langle\{p,q\}\rangle$ is the unique proper quaternionic projective subspace
passing through $p$ and $q$.  Such a subspace will be called a quaternionic (projective) line, and we will denote $\langle\{p,q\}\rangle$ by $\mathbb{L}{\{p, q\}}$. 

\subsection{Projective  Transformations}
Action of an arbitrary element $ \gamma \in \mathrm{GL}(3,\H)$ on  $ \P^2_\H$ is given by
$$ \gamma (\P(z)) = \P( \gamma(z)) \hbox{ for all } z \in \H^{3}\setminus \{0\}.$$
Note that for any non-zero  $r \in \R$,  we have $ (r\gamma) (\P(z)) = \gamma (\P(z))  \hbox{ for all } \, z \in \H^{3}\setminus \{0\}.$
We denote the corresponding quotient map by $\pi : \mathrm{GL}(3,\H) \longrightarrow  \mathrm{PGL}(3,\H)$. 
The group of projective transformations of $\P^2_\H$ may be identified with
$$\mathrm{PSL}(3,\H) := \mathrm{SL}(3,\H)/ \mathcal{Z}(\mathrm{SL}(3,\H)),$$
where  the center $ \mathcal{Z}(\mathrm{SL}(3,\H)) = \{ \pm \mathrm{I}_{3}  \}$ and  it acts by the usual scalar multiplication on $\H^{3}$. Given $\tilde{\gamma} \in \mathrm{PSL}(3,\H)$, we say that $ {\gamma} \in \mathrm{SL}(3,\H)$  is a lift of $\tilde{\gamma}$ if $\pi( {\gamma}) = \tilde{\gamma}$. There are exactly two lifts  $\tilde{\gamma}, -\tilde{\gamma} \in \mathrm{SL}(3,\H)$  for each projective transformation $\tilde{\gamma}  \in \mathrm{PSL}(3,\H)$. 

Let $ \tilde{g} \in \mathrm{PSL}(3,\H)$ be a projective transformation,  then it is called elliptic, loxodromic and parabolic according as its lift $   {g} \in \mathrm{SL}(3,\H)$,  see Definition \ref{def:classification}.  Note that this classification is well defined in view of the Jordan decomposition of matrices over the quaternions, see \lemref{lem-Jordan-M(n,H)}. 

\subsection{Pseudo-projective  Transformations}
The concept of pseudo-projective transformations was introduced by P. J.
Myrbeg \cite{myb}, see \cite{bcn} for a brief explanation. This has been very useful for understanding the complex Klenian groups.
We extend this notion over the quaternions here. 

Let $\widetilde{M}: \H^{3} \longrightarrow \H^{3} $ be a non-zero (right) linear transformation which is not necessarily invertible.  Let $\mathrm{Ker}(\widetilde{M}) \subseteq \H^{3} $ be its kernel.  Then $\widetilde{M}$ induces a well-defined transformation $M : \P^2_\H \setminus \mathrm{Ker}({M}) \longrightarrow \P^2_\H$ given by $M(\P(v)) = \P (\widetilde{M}(v))$,  where 
$\mathrm{Ker}({M}) \subseteq \P^2_\H$ is the image of $\mathrm{Ker}(\widetilde{M}) \setminus \{0\}$ under the projective map $\P$.  Note that $M$ is well defined because $\widetilde{M}(v) \neq 0$,  and $M $ is a projective transformation on its domain: for every non-zero quaternion $ \alpha \in \H $,  $M (\P(v \alpha)) = \P(\widetilde{M}(v \alpha))$ and coincide with $\P (\widetilde{M}(v))$ in $\P^2_\H $.

We call the map $M $ a pseudo-projective transformation,  and we denote by $\mathrm{QP}(3, \H)$ the space of all pseudo-projective transformations of $\P^2_\H $.  Therefore,
$$ \mathrm{QP}(3, \H) = \{M \mid \widetilde{M} \hbox{ is a non-zero linear transformation of }  \H^{3}\}.$$
Note that $\mathrm{PSL}(3,\H) \subseteq \mathrm{QP}(3, \H)$.   We will use pseudo-projective transformations and classification of the elements of the lifts of $\mathrm{PSL}(3,\mathbb{H})$ (see \defref{def:classification}) for computing the limit sets for different cyclic subgroups of $\mathrm{PSL}(3,\H)$.

\subsection{Dense subsets of $S^1$ and $T^2$}
We state without proof some of the important well-known results of unit circle and torus. 
\begin{lemma} [cf.~{\cite[Proposition 1.3.3]{kat}}]	\label{prop:conv1}
Let  $S^1=\{x\in \mathbb{C} \mid  |x|=1\}$ and $\alpha\in \mathbb{R}\setminus \mathbb{Q}$. Then 
the set $A=\{e^{2\pi i  n \alpha} \in S^1 \mid  n\in \mathbb{Z}\}$ is dense in $S^1$. 
\end{lemma}

Note that for each $\alpha\in \mathbb{R}\setminus \mathbb{Q}$,  there exists a sequence $( n_k)_{k\in \mathbb{N}}$ such that $e^{2\pi i  n_k \alpha}$ converges to $ 1$ as $k\rightarrow \infty $.  Further, if $r \in \mathbb{Q}$ then we can also find a suitable sequence $( n_k)_{k\in \mathbb{N}}$ such that   $(e^{2\pi i n_k  \alpha},e^{2\pi i n_k  r \alpha})$ converges to $ (1,1)$ as $k\rightarrow \infty $. 

\begin{lemma}[cf.~{\cite[Proposition 1.4.1]{kat}}] \label{prop:conv2}
Let $T^2=S^1\times S^1$ and  $ \alpha, \beta \in \R$ such that $ \frac{\alpha}{\beta} \in \mathbb{R}\setminus \mathbb{Q}$. Then the set $A=\{(e^{2\pi i n\alpha},e^{2\pi i n\beta}) \in T^2 \mid n\in \mathbb{Z}\}$  is dense in $T^2$. 
\end{lemma}

\begin{remark} \label{rem-torus-denseness}
It follows  that if  $\alpha\in \mathbb{R}\setminus \mathbb{Q}$ and $\beta \in \R$ then there exists a sequence $(n_k)_{k\in \mathbb{N}}$ such that $(e^{2\pi i n_k\alpha},e^{2\pi i n_k\beta})\xrightarrow{k\rightarrow \infty} (1,1)$. We can also deduce that there exists a sequence $( n_k)_{k\in \mathbb{N}}$ such that $(e^{2\pi i \alpha n_k},e^{2\pi i \beta n_k},e^{2\pi i \gamma n_k})\xrightarrow{k\rightarrow \infty} (1,1,1)$ whenever $\alpha,\beta,\gamma\in \mathbb{R}\setminus\mathbb{Q} $,   see \cite{kat},  \cite{TR}.
\qed
\end{remark}

\subsection{Useful Lemmas}
In this subsection, we shall discuss two useful lemmas we need in our later sections. The proofs of these lemmas follow in a similar line of argument as in the corresponding complex analogue, see \cite[Lemma 3.2]{bcn10},  {\cite[Proposition 3.1]{can2}} and {\cite[Lemma 5.3]{TR}}.

\begin{lemma}	\label{lem:app1}
	Let $(g_n)_{n\in \mathbb{N}}\subset \mathrm{PSL}(3,\mathbb{H})$ be a sequence. Then there exists a subsequence $(g_{n_k})_{k\in \mathbb{N}}$ and an element $g\in \mathrm{QP}(3, \mathbb{H})$ such that $g_{n_k}\rightarrow g$ uniformly on the compact subsets of $\mathbb{P}^2_\mathbb{H}\setminus Ker(g)$.
\end{lemma}
\begin{proof}
	Let us take a lift $\tilde{g}_n\in\mathrm{SL}(3,\mathbb{H})$ of $g_n$, where $\tilde{g}_n=(g^{(n)}_{ij})_{1 \leq i,j \leq 3}$. Also consider the supremum norm of $g_n$ given by $|g_n|=\text{max}\{|g^{(n)}_{ij}|  \mid i,j=1,2,3 \}$. Then the sequence $\left(\frac{1}{|g_{n}|}\tilde{g}_n\right)_{n\in \mathbb{N}}$ is a bounded sequence of matrices  over  quaternions. It gives two bounded sequences of matrices over complex numbers since one can write $q\in \mathbb{H}$ as $q=q_1+q_2j$, $q_1,\;q_2\in \mathbb{C}$. Recall that every bounded sequence of complex matrices has a convergent subsequence. This implies that the sequence of matrices $\left(\frac{1}{|g_{n}|}\tilde{g}_n\right)_{n\in \mathbb{N}}$ has a  convergent subsequence, that is, there is a subsequence $\left(\frac{1}{|g_{n_k}|}\tilde{g}_{n_k}\right)_{k\in \mathbb{N}}$ which converges to some non-zero $\tilde{g}  \in \mathrm{M}(3, \mathbb{H})$. Assume that $g \in \mathrm{QP}(3, \mathbb{H})$ be the pseudo-projective transformation corresponding to $\tilde{g}$.	
	
	Now,  let $K\subset \mathbb{P}^2_\mathbb{H}\setminus Ker(g)$ be a compact set and $\widetilde{K}=\{k\in \mathbb{H}^3 \mid  \P(k)\in K\}\cap \{u\in \mathbb{H}^3 \mid |u|=1\}$. Then $\P (\widetilde{K})=K$ and $\frac{1}{|g_{n_k}|}\tilde{g}_{n_k}$ converges to $\tilde{g}$ on $\tilde{K}$ in the compact-open topology. As for each $k \in \N$,  $\frac{1}{|g_{n_k}|}\tilde{g}_{n_k}$ is a lift of $g_{n_k}$,  hence $g_{n_k}\rightarrow g$ on $K$ uniformly. 
\end{proof}

The complex version of the following lemma was introduced by Navarrete in  \cite{TR}. 

\begin{lemma}\label{lemma:L2}
	Let	$G$ be a subgroup of $\mathrm{PSL}(3,\mathbb{H})$. If $C \subset \mathbb{P}^2_\mathbb{H}$ is a closed subset such that for every compact subset $K\subset \mathbb{P}^2_\mathbb{H}\setminus C$, the cluster points of the family  $\{g(K)\}_{g\in G}$ of compact sets are contained in $L_0(G)\cup L_1(G)$, then $L_2(G)\subseteq C$.
\end{lemma}

\begin{proof}
	Suppose there is a point $x\in L_2(G)\setminus C$. Then there exists a compact subset  $K\subset \mathbb{P}^2_\mathbb{H}\setminus L_0(G)\cup L_1(G)$ such that $x$ is a cluster point of the orbit $\{g(K)\}_{g\in G}$. This implies that there exists a sequence $(g_n)_{n\in\mathbb{N}} \subset G$ such that $g_n(k_n)\xrightarrow{n\rightarrow \infty} x$, where $(k_n)_{n\in \mathbb{N}}\subset K$ and $k_n\xrightarrow{n\rightarrow \infty} k\in  K\subset \mathbb{P}^2_\mathbb{H}\setminus \{L_0(G)\cup L_1(G)\}$.  Since $ x \notin C$ and $C$ is a closed set, so by discarding a finite number of terms, if needed, we can assume $g_n(k_n)) \in \mathbb{P}^2_{\mathbb{H}}\setminus C$ for all $n \in \N$.  Note that $ g^{-1}_n (g_n(k_n)) = k_n$ for each $n \in \N$. Then  the hypothesis  applied to the compact subset $\{g_n(k_n)\}_ {n \in \mathbb{N} } \cup \{x\}\subset \mathbb{P}^2_{\mathbb{H}}\setminus C$ yields that $k \in L_0(G)\cup L_1(G)$,  which is a contradiction.
\end{proof}

\section{Kulkarni Limit Sets for Elliptic Transformations}\label{sec:kul-ell}
In this section, we shall find the Kulkarni limit set for the cyclic subgroups generated by the elliptic elements of $\mathrm{PSL}(3,\mathbb{H})$. Proof of Table \ref{table:1}  given in section \ref{sec:intro} follows from the next theorem.
\begin{theorem}	\label{th:elliptic}
	Let $\tilde{g}\in \mathrm{PSL}(3,\mathbb{H})$ be an elliptic transformation given by \\
	$g=\left(\begin{array}{ccc}	e^{2\pi i \alpha} & 0 & 0\\0 & e^{2\pi i \beta} & 0\\ 0 & 0 & e^{2\pi i \gamma}
	\end{array}\right), $ where $\alpha, \beta, \gamma \in \mathbb{R}.$ The Kulkarni sets $L_0(G),\;L_1(G),\; L_2(G)$ for cyclic subgroup $G:=\langle \tilde{g} \rangle $ are as follows:
	\begin{enumerate}[(i)]
		\item  \medskip Rational elliptic: If  $\alpha,\beta,\gamma\in \mathbb{Q}$, then $L_0(G)=L_1(G)=L_2(G)=\phi$.
		\item  	\medskip Simple irrational elliptic: If $\alpha=\beta=\gamma\in \mathbb{R}\setminus\mathbb{Q}$, then $L_0(G)=\mathbb{P}^2_\mathbb{C}$, $L_1(G)=\mathbb{P}^2_\mathbb{H}$ and $L_2(G)=\phi$.
		\item  \medskip Compound irrational elliptic: If $\alpha,\beta,\gamma$ not all are rational and all are not equal, then $L_0(G)=\{x\in \mathbb{P}^2_\mathbb{H} \mid x \text{ is a fixed point}\}$, $L_1(G)=\mathbb{P}^2_\mathbb{H}$ and $L_2(G)=\phi$.
	\end{enumerate}
\end{theorem} 
\begin{proof}
	We can divide our calculations into three parts depending on the rationality of $\alpha,\beta$ and $\gamma$. 
	\begin{enumerate}[(i)]
		\item   \medskip \textbf{Rational elliptic:} Let $\alpha, \beta, \gamma \in \mathbb{Q}$. Then there exists a $n_0 \in \mathbb{N}$ such that $e^{2\pi i n_0 \alpha}=e^{2\pi i n_0 \beta}=e^{2\pi i n_0 \gamma}=1$, that is, $G =\langle \tilde{g} \rangle $ is a finite group. Therefore, there is no $ x \in \mathbb{P}^2_\mathbb{H}$, which has an infinite isotropy group. Thus $L_0(G)=\{\phi\}$. Similarly, we have  $L_1(G)=\phi$ and $L_2(G)=\phi$. 
		
		\item   \medskip \textbf{Simple irrational elliptic:} Let $\alpha= \beta=\gamma\in \mathbb{R}\setminus \mathbb{Q}$.  	Recall that $wj = j \bar{w}$ for each $ w\in \C$.	Consider the copy of the complex projective space inside $\mathbb{P}^2_\mathbb{H}$ that we have identified as $\mathbb{P}^2_\mathbb{C}:=\{[z]_\mathbb{C} \mid  z\in \mathbb{C}^3 \setminus \{0\}\}$, where $ [z]_\mathbb{C}=\{w\in \mathbb{C}^3 \mid w= z a \hbox{ for some } a \in \mathbb{C}\setminus\{0\}\}$.
		Then $L_0(G)=\mathbb{P}^2_\mathbb{C}$, since by the definition of the projective space, for all $[x: y: z] \in \mathbb{P}^2_\mathbb{C}$  we have
		$$g^n\begin{bmatrix}
			x\\y\\z
		\end{bmatrix}=\begin{bmatrix}
			e^{2\pi ni \alpha}x\\e^{2\pi n i \beta}y\\e^{2\pi ni \gamma}z
		\end{bmatrix}=e^{2\pi ni \alpha}\begin{bmatrix}
			x\\y\\z
		\end{bmatrix} =\begin{bmatrix}
			x\\y\\z
		\end{bmatrix} e^{2\pi ni \alpha} =\begin{bmatrix}
			x\\y\\z
		\end{bmatrix}.$$
		Since	$\alpha\in \mathbb{R}\setminus \mathbb{Q}$, there exists a sequence $(n_k)_{k\in \mathbb{N}}\subset \mathbb{N}$ such that $e^{2\pi i n_k\alpha} \xrightarrow{k\rightarrow \infty }1$, see Lemma \ref{prop:conv1}. This implies  $g^{n_k} [x: y: z] \xrightarrow{k\rightarrow \infty} [x: y: z]$ for every point  $[x: y: z]\in \mathbb{P}^2_\mathbb{H}\setminus L_0(G)$.   Therefore,  $\mathbb{P}^2_\mathbb{H}\setminus L_0(G)\subset L_1(G)$. Again, for every point ${p}\in L_0(G)$ there exists a point $q\in \mathbb{P}^2_\mathbb{H}\setminus L_0(G)$ such that $g^{n_k}(q) \xrightarrow{k\rightarrow \infty} p$ (one can assume $q=pj$). Hence,  $L_1(G)=\mathbb{P}^2_\mathbb{H}$. Consequently,  $L_2(G)=\{\phi\}$. 
		
		\item  \medskip \textbf{Compound irrational elliptic:}
		When all the eigenvalues of an elliptic transformation are not equal and not all are rational numbers,  we call them \textit{compound irrational elliptic}.  We have the following cases.

		\begin{enumerate}
			
			\item  \medskip \textbf{Compound irrational elliptic type I:} Let $\alpha,\; \beta\in \mathbb{Q}$ and $\gamma\in \mathbb{R}\setminus \mathbb{Q}$. Note that $L_0(G)$ is the set of fixed points of $g$ as these points have an infinite isotropy group. Since $\alpha,\; \beta \in \mathbb{Q}$,  there exists a $n_0\in \mathbb{N}$ such that $e^{2\pi i n_0 \alpha}=e^{2\pi i n_0 \beta}=1$. Further,  using Lemma \ref{prop:conv1} we get a sequence $(n_k)_{k\in \mathbb{N}}$ such that $(e^{2\pi i \alpha n_k},e^{2\pi i \beta n_k},e^{2\pi i \gamma n_k})\xrightarrow{k\rightarrow \infty} (1,1,1)$.
			With the above argument in knowledge, we can prove that for every point $p \in \mathbb{P}^2_\mathbb{H}$ there exists a point $q \in \mathbb{P}^2_\mathbb{H}\setminus L_0(G)$ such that $g^{n_k}(q) \xrightarrow{k\rightarrow \infty} p$. To see this,  for each point $p = [x: y: z]\in \mathbb{P}^2_\mathbb{H}\setminus L_0(G)$, choose $q =p$, then 
			$g^{n_k} (q) 	 = g^{n_k}\begin{bmatrix} x\\y\\z\end{bmatrix}=\begin{bmatrix} e^{2\pi i  n_k\alpha}x\\e^{2\pi i  n_k\beta}y\\e^{2\pi i  n_k\gamma}z\end{bmatrix} \xrightarrow{k\rightarrow \infty} \begin{bmatrix} x\\y\\z\end{bmatrix}$.  Observe that when $p \in L_0(G)$,  then depending upon the point $p$ we can also find a suitable   point $q \in \mathbb{P}^2_\mathbb{H}\setminus L_0(G)$ such that $g^{n_k} (q) $ converges to $p$ as ${k\rightarrow \infty}$. Therefore,  we have $L_1(G)=\mathbb{P}^2_\mathbb{H}$ and consequently $L_2(G)=\{\phi\}$. 
			
			A similar line of argument will hold if, instead of $\gamma$, either $\alpha$ or $\beta$ is irrational.
			
			\item  \medskip \textbf{Compound irrational elliptic Type II:} Let $\alpha\in \mathbb{Q}$ and $\beta, \; \gamma\in \mathbb{R}\setminus \mathbb{Q}$. Now $L_0(G)$ is the set of fixed points of $g$ as these points have an infinite isotropy group. In this case we can find a sequence $(n_k)_{k\in \mathbb{N}}$ such that $e^{2\pi i n_k\beta} \xrightarrow{k\rightarrow \infty} 1$ and $e^{2\pi i n_k\gamma} \xrightarrow{k\rightarrow \infty} 1$, see  \remref{rem-torus-denseness}. Also, since $\alpha\in \mathbb{Q}$, there exists a $n_0\in \mathbb{N}$ such that $e^{2\pi i n_0 \alpha}=1$. Therefore for every $[x: y: z]\in \mathbb{P}^2_\mathbb{H}\setminus L_0(G)$ we have a sequence $(n_k)_{k\in \mathbb{N}}$ such that $g^{n_k}[x: y: z]  \xrightarrow{k\rightarrow \infty} [x: y: z]$ and so $\mathbb{P}^2_\mathbb{H}\setminus L_0(G)\subset L_1(G)$. Again for any $p\in L_0(G)$ we can choose a suitable point $q\in \mathbb{P}^2_\mathbb{H}\setminus L_0(G)$ such that $g^{n_k}(q)\rightarrow p$ as $k \rightarrow \infty$. Therefore, $L_1(G)=\mathbb{P}^2_\mathbb{H}$. Consequently,  $L_2(G)=\phi$.  
			A similar line of argument will hold if, instead of $\alpha$, either $\beta$ or  $\gamma$   is rational.
			
			\item  \medskip  \textbf{Compound irrational elliptic Type III:} Let $\alpha, \beta,\gamma\in \mathbb{R}\setminus \mathbb{Q}$ and not all are equal. Then $G$ is a infinite group, so  $L_0(G)=\{x\in \mathbb{P}^2_\mathbb{H} \mid x \text{ is a fixed point}\}$. To find $L_1(G)$, we can use  \remref{rem-torus-denseness} and show $L_1(G)=\mathbb{P}^2_\mathbb{H}$. Consequently,  $L_2(G)=\phi$.
		\end{enumerate}
	\end{enumerate}
	This completes the proof.
\end{proof}

\section{Kulkarni Limit Sets for Loxodromic Transformations}
\label{sec:kul-lox}
In this section,  we consider loxodromic transformations of $\mathrm{SL}(3,\mathbb{H})$,  see Definition \ref{def:classification}.   The following theorem proves Table \ref{table:2}.
\begin{theorem}
	Let $\tilde{g}\in \mathrm{PSL}(3,\mathbb{H})$ be a loxodromic transformation and $g\in \mathrm{SL}(3,\mathbb{H})$ be a  lift of $\tilde{g}$. Then the Kulkarni sets for $G=\langle \tilde{g}\rangle$ are the following:
	\begin{enumerate}[(i)]
		\item \medskip Regular loxodromic: If $g=\begin{pmatrix}
			\lambda & 0 & 0\\0 & \mu & 0\\0 & 0 & \xi
		\end{pmatrix}$ with $|\lambda|<|\mu|<|\xi|$ then $L_0(G)= L_1(G) = \{e_1,e_2,e_3\}$ and $L_2(G)=\mathbb{L}\{e_1,e_2\}\cup\mathbb{L}\{e_2,e_3\}$.
		\item  \medskip Screw loxodromic: If $g=\begin{pmatrix}
			\lambda & 0 & 0\\0 & \mu & 0\\0 & 0 & \xi
		\end{pmatrix}$ with $|\lambda|=|\mu|\neq 1$ such that $\lambda \neq \mu$ and $|\xi|=\frac{1}{|\lambda|^2}$ then $L_0(G)=\{e_1,e_2,e_3\}$ and $L_1(G)=\mathbb{L}\{e_1,e_2\}\cup\{e_3\}=L_2(G)$.
		\item  \medskip Homothety: If $g=\begin{pmatrix}
			\lambda & 0 & 0\\0 & \lambda & 0\\0 & 0 & \xi
		\end{pmatrix}$ with $|\lambda|\neq 1$ and $|\xi|=\frac{1}{|\lambda|^2}$ then $L_1(G)= L_2(G)=\mathbb{L}\{e_1,e_2\}\cup\{e_3\}$ and  $L_0(G)=  \{e_3\} \cup 
		\begin{cases}
			\mathbb{L}\{e_1,e_2\}, & \text{if $\lambda \in \R$}\\
			\mathbb{L}_{\C}\{e_1,e_2\}, & \text{  if $ \lambda \in \C \setminus \R $.}
		\end{cases}
		$
		\item \medskip  Loxo-parabolic: If $g=\begin{pmatrix}
			\lambda & 1 & 0\\0 & \lambda & 0\\0 & 0 & \xi
		\end{pmatrix}$ with $|\lambda|\neq 1$ and $|\xi|=\frac{1}{|\lambda|^2}$ then $L_0(G)=\{e_1,e_3\} = L_1(G)$ and $L_2(G)=\mathbb{L}\{e_1,e_2\}\cup\mathbb{L}\{e_1,e_3\}$.
	\end{enumerate}
\end{theorem}
\begin{proof}
	We will consider various cases depending on the eigenvalues to determine the Kulkarni sets.
	\begin{enumerate}[(i)]
		\item  \textbf{Regular loxodromic}:  First, we will find $L_0(G)$. A point $p =[x:y:z] \in \mathbb{P}^2_\mathbb{H}$ has an infinite isotropy group if $g^n(p)=p$ for infinitely many $n\in \mathbb{Z}$ and $g^n(p)=\left(\begin{array}{ccc}
			\lambda^n & 0 & 0\\0 & \mu^n & 0\\0 & 0 & \xi^n
		\end{array}\right)\begin{bmatrix}
			x\\y\\z
		\end{bmatrix}=\left[\begin{array}{c}
			\lambda^n x\\ \mu^n y\\ \xi^n z
		\end{array}\right]=\left[\begin{array}{c}
			x\\  y\\  z
		\end{array}\right]$ is true only when $p=e_1,e_2$ or $e_3$ (as $\mathbb{P}^2_\mathbb{H}$ is a projective space). Hence, $L_0(G)=\{e_1,e_2,e_3\}$. 
		
		For calculating $L_1(G)$,  we have to find the set of the cluster points of orbits of points in $\mathbb{P}^2_\mathbb{H}\setminus L_0(G)$. Recall that for each non-zero $r \in \R$, $r g $ is a lift of $g$, and both induce the same projective transformation on  $\mathbb{P}^2_\mathbb{H}$.
		Now consider 
		$|\xi|^{-n}g^{n}=\left(\begin{array}{ccc}
			|\xi|^{-n}|\lambda|^n e^{in\arg(\lambda)} & 0 & 0\\
			0 & |\xi|^{-n}|\mu|^n e^{ in\arg(\mu)} & 0\\
			0 & 0 & e^{ i n\arg (\xi)}
		\end{array}\right)$, where $n \in \N$. 
		Note that there is a  subsequence $(n_k)_{k\in \mathbb{N}} \subset \N$ such that $ e^{ i n_{k} \arg (\xi)} \xrightarrow{k\rightarrow \infty} 1$.
		Since  $\frac{|\lambda|}{|\xi|} <1$ and $\frac{|\mu|}{|\xi|} <1$, there exists a subsequence  of $\left(|\xi|^{-n}g^{n}\right)_{n\in \mathbb{N}}$, still denoted by $\left(|\xi|^{-n}g^{n}\right)_{n\in \mathbb{N}}$,   which converges to $D_1 = \begin{pmatrix}
			0 & 0 & 0\\0 & 0 & 0\\0 & 0 & 1
		\end{pmatrix}$ as $n\rightarrow \infty$.  
		Note that $  Ker(D_1) = \mathbb{L}{ \{ e_1,e_2 \}}= \{ [x:y:z] \in  \mathbb{P}^2_\mathbb{H} \mid z =0    \}$.
		Thus for every point $ p \in \mathbb{P}^2_\mathbb{H}\setminus \mathbb{L} { \{e_1,e_2\}}$,  $\left(|\xi|^{-n}g^{n}(p)\right)_{n\in \mathbb{N}}$ converges to $e_{3}  \in \mathbb{P}^2_\mathbb{H}$ as $n \rightarrow \infty$.  Hence, $e_3 \in L_1(G)$. Similarly, by considering a lift $|\lambda|^{n} g^{-n}$ of $g^{-n}$, we can show that there is a subsequence of $\left(|\lambda|^{n} g^{-n}\right)_{n\in \mathbb{N}} $, which converges to $D_2 = \mathrm{diag}(1,0,0)$ as ${n\rightarrow \infty}$. Note that $  Ker(D_2) = \mathbb{L}{ \{ e_2,e_3 \}}= \{ [x:y:z] \in  \mathbb{P}^2_\mathbb{H} \mid x =0    \}$. Thus for every point $ p \in \mathbb{P}^2_\mathbb{H}\setminus  \mathbb{L} { \{e_2,e_3\}}$,  $\left(|\lambda|^{n} g^{-n}(p)\right)_{n\in \mathbb{N}} \xrightarrow{n\rightarrow \infty}  e_{1}  \in \mathbb{P}^2_\mathbb{H}$ and hence $e_1 \in L_1(G)$. Now consider the sequences $(|\mu|^{-n}g^{n})_{n\in \N}$  and  $(|\mu|^{n}g^{-n})_{n\in \N}$.  Then  there are subsequences of $|\mu|^{-n}g^{n} (p) $ and $|\mu|^{n}g^{-n}(q)$ that converge to $e_2 \in \mathbb{P}^2_\mathbb{H}$ whenever $p \in \mathbb{L}{ \{ e_1,e_2 \}}$ and $ q \in \mathbb{L}{ \{ e_2,e_3 \}}$.  Therefore, $e_2 \in  L_1(G)$. Also, from the above discussion, it follows that for every point $ p \in \mathbb{P}^2_\mathbb{H}\setminus L_0(G)$,   the set of cluster points of the orbits $\{g^n(p)\}_{n\in \Z}$ lies in the set $ \{e_1,e_2,e_3\}$. Consequently, $L_1(G)=\{e_1,e_2,e_3\}$.
		
		To find $ L_2(G)$, consider  a closed set  $C=\mathbb{L}{\{e_1,e_2\}} \cup  \mathbb{L}{\{e_2,e_3\}}$  in $\mathbb{P}^2_\mathbb{H}$. Now note that $D_1$ and  $D_2$ induce  pseudo-projective transformations in $\mathrm{QP}(3,\H)$ such that $\mathbb{L}{ \{ e_1,e_2 \}}= Ker(D_1)$ and $\mathbb{L}{ \{ e_2,e_3\}}= Ker(D_2) $. Therefore, using Lemma \ref{lem:app1}  we can show that the set of cluster points of orbits of compact subsets of  $\mathbb{P}^2_\mathbb{H}\setminus C$ lies in  $L_0(G)\cup L_1(G)=\{e_1,e_2,e_3\}$.  Then, using Lemma \ref{lemma:L2} we have  $ L_2(G)\subset C$, that is, $ L_2(G) \subset
		\mathbb{L} {\{e_1,e_2\}}\cup\mathbb{L}{\{e_2,e_3\}}$. Further,	to show  that $\mathbb{L}{\{e_2,e_3\}} \subset  L_2(G)$,  choose a compact set $K=\{[x:y:z] \in \mathbb{P}^2_\mathbb{H} \mid  |x|^2 =|y|^2 +|z|^2 \}$ which is a subset of $ \mathbb{P}^2_\mathbb{H}\setminus L_0(G)\cup L_1(G)$.  For a point $[0 : y:z]\in \mathbb{L}\{e_2,e_3\}$,  take a sequence $(k_n)_{n\in \N}\subset K$ such that $k_n=\left[(|\xi^n y|^2+|\mu^n z|^{2})^{1/2} :|\xi|^n y : |\mu|^n z\right] $. Then for each $n\in \N$, by considering  lift  $|\mu \xi|^{-n}g^n$ of $g^n$, we have 
		$$g^n(k_n) =  |\mu \xi|^{-n}g^n(k_n) = \left[ \frac{|\lambda|^n e^{in\arg(\lambda)} }{|\mu \xi|^{n}} (|\xi^n y|^2+|\mu^n z|^{2})^{1/2}: e^{in\arg(\mu)} y: e^{in\arg(\xi)}  z \right]$$
		$$= \left[ \frac{|\lambda|^n e^{in\arg(\lambda)}}{|\mu|^{n}} (|y|^2+ \frac{|\mu|^{2n}}{|\xi|^{2n}}|z|^{2})^{1/2}: e^{in\arg(\mu)}y: e^{in\arg(\xi)} z \right].$$
		Since	$|\lambda|<|\mu|<|\xi|$, we have a subsequence of  $g^n(k_n)$ which converges to $ [0:y:z]$ as $n\rightarrow \infty$. Hence, $\mathbb{L}{ \{e_2,e_3\}}\subset L_2(G)$. Similarly, to show that $\mathbb{L} {\{e_1,e_2\}}\subset L_2(G)$, we take a sequence $(k'_n)_{n\in \N}$ in the compact set  $K' = \{[x:y:z] \in \mathbb{P}^2_\mathbb{H} \mid |x|^2 +|y|^2  = |z|^2 \} \subset \mathbb{P}^2_\mathbb{H}\setminus L_0(G) \cup L_1(G)$
		such that $k'_n=\left[|\mu|^{-n} x : |\lambda|^{-n} y : (|\mu^{-n} x|^2+|\lambda^{-n} y|^{2})^{1/2} \right] $. Then by considering  lift  $|\lambda\mu |^{n}g^{-n}$ of $g^{-n}$, we can show  that there is a subsequence of  $g^{-n}(k'_n) $ which converges to a point  $[x:y:0] \in \mathbb{L} {\{e_1,e_2\}}$ as $n\rightarrow \infty$.
		Hence,   $\mathbb{L}{ \{e_1,e_2\}}\cup\mathbb{L}{ \{e_2,e_3\}}\subset L_2(G)$.
		Consequently,   $L_2(G)=\mathbb{L}{ \{e_1,e_2\}}\cup\mathbb{L}{\{e_2,e_3\} }$.
		
		\item \medskip \textbf{Screw loxodromic}: 
		In this case, we have $|\lambda|=|\mu|\neq 1, \lambda \neq \mu$ and  $ |\xi|=1/{|\lambda|^2}$. Further,  we will assume that $1 < |\lambda| $ since proof for the case $ |\lambda| < 1$ follows using a similar line of argument.
		If $g^n\left[\begin{array}{c}
			x\\y\\z
		\end{array}\right]=\left[\begin{array}{c}
			\lambda^n x\\\mu^n y\\\xi^n z
		\end{array}\right]  = \left[\begin{array}{c}
			x\\y\\z
		\end{array}\right]$ then the only  points with infinite isotropy groups are $\{e_1,e_2,e_3\}$.  Hence $L_0(G)=\{e_1,e_2,e_3\}$.
		
		Consider a  sequence $\left({|\lambda|^{-n}}g^n\right)_{n\in \mathbb{N}}$ such that
		$$|\lambda|^{-n}g^n=\left(\begin{array}{ccc}
			e^{in\arg(\lambda)} & 0 & 0\\0 & e^{in\arg(\mu)} & 0\\0 & 0 & \frac{1}{|\lambda|^{3n}}e^{ in\arg(\xi)}
		\end{array}\right).$$
		Then we get a subsequence of $\left(|\lambda|^{-n} g^{n}\right)_{n\in \mathbb{N}}$ which converges to  $D_3 = \mathrm{diag}(1,1,0)$. Note that $  Ker(D_3)= \{ e_3\}=  \{ [x:y:z] \in  \mathbb{P}^2_\mathbb{H} \mid x =y=0   \} $.  Thus for every point $ p \in \mathbb{P}^2_\mathbb{H}\setminus { \{ e_3\}}$,  $\left(|\lambda|^{-n} g^{n}(p)\right)_{n\in \mathbb{N}}$  converges to a point of $ \mathbb{L}{ \{ e_1,e_2 \}}  \in \mathbb{P}^2_\mathbb{H}$ as $n \rightarrow \infty$.  Therefore, for each point $q = [x:y:0] \in \mathbb{L}{ \{ e_1,e_2 \}}$, we can choose a point   $p = [x:y:1] \in   \mathbb{P}^2_\mathbb{H}\setminus L_0(G)$ such that $g^n(p) \xrightarrow{n\rightarrow \infty} q $ and hence $q \in  L_1(G)$. Consequently, $ \mathbb{L}{ \{ e_1,e_2 \}} \subset  L_1(G)$.  Now consider the sequence $\left(|\xi|^{n}g^{-n}\right)_{n\in \mathbb{N}}$ such that  $$|\xi|^{n}g^{-n} = \mathrm{diag}\left(\frac{1}{|\lambda|^{3n}}e^{-in\arg(\lambda)}, \frac{1}{|\lambda|^{3n}}e^{-i n\arg(\mu)}, e^{-in\arg(\xi)}\right).$$
		Since $1 <|\lambda| $, we get a subsequence	of   $\left(|\xi|^{n}g^{-n}\right)_{n\in \mathbb{N}}$ which converges to $D_1 = \mathrm{diag}(0,0,1)$, where $  Ker(D_1)= \mathbb{L}{ \{ e_1,e_2 \}}$. Thus for every point $ p \in \mathbb{P}^2_\mathbb{H}\setminus \mathbb{L}{ \{ e_1,e_2 \}}$,  $\left(|\xi|^{n}g^{-n}(p)\right)_{n\in \mathbb{N}}$  converges to $ e_3  \in \mathbb{P}^2_\mathbb{H}$ as $n \rightarrow \infty$. Therefore, $ e_3 \in  L_1(G)$. Also, from the above discussion, it follows that for every point $ p \in \mathbb{P}^2_\mathbb{H}\setminus L_0(G)$,   the set of cluster points of the orbits $\{g^n(p)\}_{n\in \Z}$ lies in the set $ \mathbb{L}{ \{e_1,e_2\}} \cup \{e_3\}$. Consequently, $L_1(G)=\mathbb{L}{ \{e_1,e_2\}} \cup \{e_3\}$.

		To find $L_2(G)$, consider a closed set $C=\mathbb{L}{\{e_1,e_2\}}\cup\{e_3\} \subset \mathbb{P}^2_\mathbb{H} $. Now note that $D_1$ and  $D_3$ induce pseudo-projective transformations in $\mathrm{QP}(3,\H)$ such that $\mathbb{L}{ \{ e_1,e_2 \}}= Ker(D_1)$ and $\mathbb{L}{ \{e_3\}}= Ker(D_3) $. Therefore, using Lemma \ref{lem:app1}  we can show that the set of cluster points of orbits of compact subsets of  $\mathbb{P}^2_\mathbb{H}\setminus C$ lies in  $L_0(G)\cup L_1(G)= \mathbb{L}{\{e_1,e_2\}}\cup\{e_3\}$. Then, using Lemma \ref{lemma:L2} we have  that  $ L_2(G) \subset 
		\mathbb{L} {\{e_1,e_2\}}\cup \{e_3\}$. Further,	we have already observed in the calculation of $L_1(G)$ that  points of $\mathbb{P}^2_\mathbb{H}  \setminus \mathbb{L} {\{e_1,e_2\}} \cup \{e_3\}$  converges to a point of  $\mathbb{L} {\{e_1,e_2\}}$ for positive iterates of $g$,  and converges to  $\{e_3\}$ for negative iterates of $g$. Therefore, we can easily show that $\mathbb{L}{ \{e_1,e_2\}} \cup \{e_3\}\subset L_2(G)$.  Hence, $ L_2(G)= 
		\mathbb{L}{ \{e_1,e_2\}} \cup \{e_3\}$.

		\item \medskip  \textbf{Homothety}:
			In this case, we have $ |\lambda|\neq 1$ and $ |\xi|=1/{|\lambda|^2}$.  Let $\lambda = re^{i  \arg(\lambda)},$ where $r  \neq 1$. 
	Since each point $[x:y:z] \in L_0(G)$ has an infinite isotropy group and  satisfy the  equation $g^n ([x:y:z]) =[	\lambda^n x: 	\lambda^n y:  \xi^nz] =[x:y:z]$ 
	for infinitely many values of $n \in \Z$. It is easy to note that $e_3\in L_0(G)$. 
	Recall that $wj = j \bar{w}$ for each $ w\in \C$. Thus if $\lambda =r$, then $g^n[x:y:0] =  [x r^n :y r^n :0]  = [x:y:0] $ for each $x,y \in \H$. If $\lambda \neq r$ then,  $g^n[x:y:0] =  [x \lambda^n :y \lambda ^n :0]  = [x:y:0] $ for each $x,y \in \C$. 
	Therefore, 	
$$L_0(G)=  \{e_3\} \cup 
	\begin{cases}
		\mathbb{L}\{e_1,e_2\}, & \text{if $\lambda \in \R$}\\
		\mathbb{L}_{\C}\{e_1,e_2\}, & \text{  if $ \lambda \in \C \setminus \R $}
	\end{cases}
	$$ where  $	\mathbb{L}_{\C}\{e_1,e_2\}$ denotes the complex projective line joining $e_1$ and $ e_2$.
	The proof of  $L_1(G)=\mathbb{L}{ \{e_1,e_2\}}\cup \{e_3\} =L_2(G)$ follows from a similar  line of argument as in the case of \textit{screw loxodromic}, and we omit it.
	
		\item  \medskip \textbf{Loxo-parabolic:} In the loxo-parabolic case $g=\left(\begin{array}{ccc}
			\lambda & 1 & 0\\0 & \lambda & 0\\0 & 0 & \xi
		\end{array}\right), \; |\xi|=|\lambda|^{-2}\neq 1$. Further, we will assume that $1 < |\lambda| $ since proof for the case $ |\lambda| < 1$ follows using a similar line of argument.   Note the following equation
		$$g^n\left[\begin{array}{c}
			x\\y\\z
		\end{array}\right] =   \left(\begin{array}{ccc}
			\lambda^n & n\lambda^{n-1} & 0\\0 & \lambda^n & 0\\0 & 0 & \xi^n
		\end{array}\right)\left[\begin{array}{c}
			x\\y\\z
		\end{array}\right]= \left[\begin{array}{c}
			\lambda^nx+n\lambda^{n-1}y \\ \lambda^ny\\ \xi^n z
		\end{array}\right]=\left[\begin{array}{c}
			x\\y\\z
		\end{array}\right]. $$
		Clearly, the above equation is satisfied by infinitely
		many values of $n \in \Z$ only when  $[x:y:z] = e_1$ or $e_3$. Hence, $L_0(G)=\{e_1,e_3\}$. 
		Now consider lifts $\frac{1}{n|\lambda|^{n-1}}g^n$ of $g^n$ and $|\xi|^n g^{-n}$ of $ g^{-n}$ for each $n \in \N$ such that
		$$\frac{1}{n|\lambda|^{n-1}}g^n=\left(\begin{array}{ccc}
			\frac{1}{n}|\lambda|e^{ in\arg(\lambda)} & e^{  i(n-1)\arg(\lambda)} & 0\\
			0 & \frac{1}{n}|\lambda|e^{in\arg(\lambda)} & 0\\
			0 & 0 & \frac{1}{n|\lambda|^{3n-1}}e^{ in\arg(\xi)}
		\end{array}\right)  \hbox{ and } $$
		$$|\xi|^n g^{-n}=\left(\begin{array}{ccc}
			\frac{1}{|\lambda|^{3n}}e^{- in\arg(\lambda)} & \frac{-n}{|\lambda|^{3n+1}}e^{- i(n+1)\arg(\lambda)} & 0\\
			0 & \frac{1}{|\lambda|^{3n}}e^{- in\arg(\lambda)} & 0\\
			0 & 0 & e^{- in\arg(\xi)}
		\end{array}\right).$$ 
		Then there is a subsequence of  $(\frac{1}{n|\lambda|^{n-1}}g^n)_{n\in \N}$ which converges to  $D_4 =\left(\begin{array}{ccc}
			0 & 1 & 0\\0 & 0 & 0\\0 & 0 & 0
		\end{array}\right)$, where $ Ker(D_4) = \mathbb{L} {\{e_1,e_3\} } $. Similarly, there exists a subsequence of  $(|\xi|^n g^{-n})_{n\in \N}$ which converges to $D_1= \mathrm{diag}(0,0,1)  $ where $ Ker(D_1)=\mathbb{L}{ \{e_1,e_2\}} $. Thus for every point $ p \in \mathbb{P}^2_\mathbb{H}\setminus \mathbb{L}{ \{ e_1,e_3 \}}$ and $ q \in \mathbb{P}^2_\mathbb{H}\setminus \mathbb{L}{ \{ e_1,e_2 \}}$, we have  $(\frac{1}{n|\lambda|^{n-1}}g^n(p))_{n\in \N} \xrightarrow{n\rightarrow \infty} e_1 \in \mathbb{P}^2_\mathbb{H}$ and   $(|\xi|^n g^{-n}(q))_{n\in \N} \xrightarrow{n\rightarrow \infty} e_3 \in \mathbb{P}^2_\mathbb{H}$. Therefore, $\{e_1,e_3\} \subset L_1(G)$.  Also, from the above discussion, it follows that for every point $ p \in \mathbb{P}^2_\mathbb{H}\setminus L_0(G)$,   the set of cluster points of the orbits $\{g^n(p)\}_{n\in \Z}$ lies in the set $ \{e_1,e_3\}$ and hence  $L_1(G) = \{e_1,e_3\}$.  
		
		To find $ L_2(G)$, consider  a closed set  $C=\mathbb{L}{\{e_1,e_2\}} \cup  \mathbb{L}{\{e_2,e_3\}}$  in $\mathbb{P}^2_\mathbb{H}$. Now note that $D_1$ and  $D_4$ induce  pseudo-projective transformations in $\mathrm{QP}(3,\H)$ such that $\mathbb{L}{ \{ e_1,e_2 \}}= Ker(D_1)$ and $\mathbb{L}{ \{ e_1,e_3\}}= Ker(D_2) $. Therefore, using Lemma \ref{lem:app1}  we can show that the set of cluster points of orbits of compact subsets of  $\mathbb{P}^2_\mathbb{H}\setminus C$ lies in  $L_0(G)\cup L_1(G)=\{e_1,e_3\}$.  Then, using Lemma \ref{lemma:L2} we have  $ L_2(G)\subset C$, that is, $ L_2(G) \subset
		\mathbb{L} {\{e_1,e_2\}}\cup\mathbb{L}{\{e_1,e_3\}}$. Further,	to show  that $\mathbb{L}{\{e_1,e_3\}} \subset  L_2(G)$,   consider a sequence $(k_n)_{n \in \N}$ such that $k_n = \Big[1:1:\frac{-n}{|\lambda|^{3n+1}w}\Big],$ where $w$ is a non-zero quaternion. Note that as  $1 < |\lambda | $,  $k_n$ is a sequence in the compact subset $K = \{k_n \mid n\in \mathbb{N}\} \cup \{[1:1:0]\}$ of $\mathbb{P}^2_\mathbb{H}\setminus \{e_1,e_3\}$, where $ \{e_1, e_3\} = L_0(G) \cup L_1(G)$.  Then we have 
		$$ g^{-n}(k_n) =  g^{-n} \left(\left[1:1:\frac{-n}{|\lambda|^{3n+1}w}\right]\right) =  \left[\lambda^{-n}-n \lambda^{-n-1}: \lambda^{-n} :\frac{-n \xi^{-2n}}{|\lambda|^{3n+1}w}\right].$$
		This implies, $ \frac{|\lambda|^{n+1}}{n}g^{-n}(k_n) = \left[\frac{|\lambda| e^{ -in\arg(\lambda)}}{n} - e^{ -i(n+1)\arg(\lambda)}: \frac{|\lambda| e^{ -in\arg(\lambda)}}{n} :\frac{-e^{ -i2n\arg(\xi)}}{w}\right].$
		Therefore,  there exists a subsequence of $(g^{-n}(k_n))_{n \in \N}$ which converges to a point $[w:0:1] \in \mathbb{L}{\{e_1,e_3\}}$. Thus $\mathbb{L}{\{e_1,e_3\}} \subset L_2(G)$.

		Now  consider a sequence $(k_n)_{n \in \N}$ such that $k_n = \left[\frac{e^{ -i\arg(\lambda)}}{|\lambda|} +\frac{y}{n} :\frac{1}{n}:\frac{1}{|\lambda|}\right],$ where $y$ is a non-zero quaternion. Note that  $k_n$ is a sequence in the compact subset $K = \{k_n \mid n\in \mathbb{N}\} \cup \{[e^{ -i\arg(\lambda)}:0:1]\}$ of $\mathbb{P}^2_\mathbb{H}\setminus L_0(G) \cup L_1(G)$. Then we have 
		\begin{equation*}
		g^{n}(k_n) =  \frac{1}{|\lambda|^{n}}g^{n} \left(\left[\frac{e^{ -i\arg(\lambda)}}{|\lambda|} +\frac{y}{n} :\frac{1}{n}:\frac{1}{|\lambda|}\right]\right) = \left[\frac{e^{ in\arg(\lambda)}}{n}y :\frac{e^{ in\arg(\lambda)}}{n}:\frac{1}{|\lambda|^{3n+1}}e^{ in\arg(\xi)}\right].
		\end{equation*}
		This implies, $ g^{n}(k_n) = \frac{n}{|\lambda|^{n}}g^{n}(k_n) = \Big[e^{ in\arg(\lambda)}y :e^{ in\arg(\lambda)}:\frac{n}{|\lambda|^{3n+1}}e^{ in\arg(\xi)}\Big].$
		Therefore,  there exists a subsequence of $(g^{n}(k_n))_{n \in \N}$ which converges to a point $[y:1:0] \in \mathbb{L}{\{e_1,e_2\}}$. Thus $\mathbb{L}{\{e_1,e_2\}} \subset L_2(G)$. Hence,  $L_2(G) =\mathbb{L}{ \{e_1,e_2\}}\cup\mathbb{L} {\{e_1,e_3\} }$.
\end{enumerate} 
This completes the proof.
\end{proof}

\section{Kulkarni Limit Sets for Parabolic Transformations}\label{sec:kul-par}
In this section, we work out the Kulkarni limit sets for the cyclic groups generated by the parabolic elements of $\mathrm{PSL}(3,\mathbb{H})$. We have several subcases of parabolic elements, such as unipotent, ellipto-parabolic and ellipto-translation. The following theorem proves Table \ref{table:3}.
\begin{theorem}\label{th:parabolic}
	Let $\tilde{g}\in \mathrm{PSL}(3,\mathbb{H})$ be a parabolic transformation and $g\in\mathrm{SL}(3,\mathbb{H})$ be a lift of $\tilde{g}$. Then the Kulkarni sets for $G=\langle \tilde{g}\rangle$ are given by the following:
	\begin{enumerate}[(i)]
		\item \medskip Vertical translation: If $g=\begin{pmatrix}
			1 & 1 & 0\\0 & 1 & 0\\0 & 0 & 1
		\end{pmatrix},$ then $L_0(G)=\mathbb{L}\{e_1,e_3\},\; L_1(G)=\{e_1\}$ and $L_2(G)=\{e_1\}$.
		\item \medskip Non-vertical translation: If $g=\begin{pmatrix}
			1 & 1 & 0\\0 & 1 & 1\\0 & 0 & 1
		\end{pmatrix},$ then $L_0(G)=\{e_1\},\; L_1(G)=\{e_1\}$ and $L_2(G)=\mathbb{L}\{e_1,e_2\}$.
		\item \medskip Ellipto-parabolic: If $g=\begin{pmatrix}
			e^{2\pi i\alpha} & 1 & 0\\0 & e^{2\pi i\alpha} & 0\\0 & 0 & e^{2\pi i\beta}
		\end{pmatrix}, \; 	e^{2\pi i\alpha} \neq 1,$ then $ L_2(G)=\{e_1\}$ and
		\begin{enumerate}
			\item  	\medskip Rational ellipto-parabolic: if $\alpha, \beta \in \mathbb{Q}$ then $L_0(G)=\mathbb{L}\{e_1,e_3\}$, $L_1(G)=\{e_1\}$.
			\item  	\medskip Irrational ellipto-parabolic:  if either or both of the $\alpha$ and $\beta$ are in  $\mathbb{R}\setminus \mathbb{Q}$,   then $L_0(G)=\{e_1,e_3\}$ and $L_1(G)=\mathbb{L}\{e_1,e_3\}$.
		\end{enumerate}
		\item  \medskip Ellipto-translation: If $g=\begin{pmatrix}
			e^{2\pi i\alpha} & 1 & 0\\0 & e^{2\pi i\alpha} & 1\\0 & 0 & e^{2\pi i\alpha}
		\end{pmatrix},$ then $L_0(G)=L_1(G)=\{e_1\}$ and $L_2(G)=\mathbb{L}\{e_1,e_2\}$.
	\end{enumerate}
\end{theorem}
\begin{proof}
	If all the eigenvalues of the transformation are $1$, then it is called a \textbf{unipotent}. Unipotent elements of $\mathrm{SL}(3,\mathbb{H})$ are divided into two cases depending upon their minimal polynomials. A unipotent element is called \textit{vertical}, resp.  \textit{non-vertical translations} if the corresponding minimal polynomial is $(x-1)^2$, resp. $(x-1)^3$.
	\begin{enumerate}[(i)]
		\item  	 \textbf{Vertical translation:} In this case,  consider the following equation $$g^n\begin{bmatrix}
			x\\y\\z
		\end{bmatrix}=\left(\begin{array}{ccc}
			1 & n & 0\\0 & 1 & 0\\ 0 & 0 & 1
		\end{array}\right)\begin{bmatrix}
			x\\y\\z
		\end{bmatrix}=\begin{bmatrix}
			x+ny\\y\\z
		\end{bmatrix}=\begin{bmatrix}
			x\\y\\z
		\end{bmatrix},$$ which gives us $y=0$.  This implies  $L_0(G)=\mathbb{L}{\{e_1,e_3\}}=\{[x: y: z]\in \mathbb{P}^2_\mathbb{H} \mid y=0\}$.   For the calculation of $L_1(G)$, consider the sequences $\left(\frac{1}{n}g^n\right)_{n\in \mathbb{N}}$ and  $\left(\frac{-1}{n}g^{-n}\right)_{n\in \mathbb{N}}$. 
		Then for every point $p=[x: y: z]\in \mathbb{P}^2_\mathbb{H}\setminus L_0(G)$ we have $$g^n\begin{bmatrix}
			x\\y\\z
		\end{bmatrix}= \frac{1}{n}g^n\begin{bmatrix}
			x\\y\\z
		\end{bmatrix}=\begin{bmatrix}
			\frac{1}{n}x+y\\\frac{1}{n}y\\\frac{1}{n}z
		\end{bmatrix} \xrightarrow{n \to \infty}  \begin{bmatrix}
			y\\0\\0
		\end{bmatrix} =  \begin{bmatrix}
			1\\0\\0
		\end{bmatrix}  .$$ 
		Similarly, we can show $ \left(\frac{-1}{n}g^{-n}(p)\right)_{n\in \mathbb{N}} \xrightarrow{n \to \infty} e_1$ for every point $p \in \mathbb{P}^2_\mathbb{H}\setminus L_0(G)$. Hence,  $L_1(G)=e_1$.
		Note that for each $n \in \N$, $\frac{1}{n}g^n$ and $\frac{-1}{n}g^{-n}$  are lifts of  $g^{n}$ and $g^{-n}$, respectively. Then in view of Lemma \ref{lem:app1},  we have that both the sequences $g^{n}$ and $g^{-n}$ converges to a pseudo-projective transformation $D_4 = \begin{pmatrix}
			0 & 1 & 0\\0 & 0 & 0 \\0 & 0 & 0
		\end{pmatrix} \in \mathrm{QP}(3,\H)$   uniformly on the compact subsets of $\mathbb{P}^2_\mathbb{H}\setminus \mathbb{L} { \{e_1,e_3\}}$, where $ \mathbb{L}{ \{ e_1,e_3 \}}= Ker(D_4) = \{ [x:y:z] \in  \mathbb{P}^2_\mathbb{H} \mid y =0    \}$. 	Since $\mathbb{L}{\{e_1,e_3\}}=L_0(G)\cup L_1(G)$, it follows that $e_1=[1: 0: 0]\in \mathbb{P}^2_\mathbb{H}$  is the only cluster point of the orbits   $\{g^n(K)\}_{n\in \mathbb{Z}}$    for every compact subsets $K$ of  $ \mathbb{P}^2_\mathbb{H}\setminus L_0(G)\cup L_1(G)$. As a consequence $L_2(G)=\{e_1\}$.
		
		\item \medskip  \textbf{Non-vertical translation:} In this case, we have
		$g=\left(\begin{array}{ccc}
			1 & 1 & 0\\0 & 1 & 1\\ 0 & 0 & 1
		\end{array}\right)$. 
		Then for each $n \in \N$, we have
		$g^n\begin{bmatrix}
			x\\y\\z
		\end{bmatrix}=  \left(\begin{array}{ccc}
			1 & n & \frac{n(n-1)}{2}\\0 & 1 & n\\ 0 & 0 & 1
		\end{array}\right)    \begin{bmatrix}
			x\\y\\z
		\end{bmatrix} =   \begin{bmatrix}
			x+ny+\frac{n(n-1)}{2}z\\y+nz\\z
		\end{bmatrix}$ and 
		$ g^{-n}\begin{bmatrix}
			x\\y\\z
		\end{bmatrix} =     	\left(\begin{array}{ccc}
			1 & -n & \frac{n(n+1)}{2}\\0 & 1 & -n\\ 0 & 0 & 1
		\end{array}\right)   \begin{bmatrix}
			x\\y\\z
		\end{bmatrix}     =\begin{bmatrix}
			x-ny+\frac{n(n+1)}{2}z\\y-nz\\z
		\end{bmatrix} $. Therefore, if the equation $g^n(p)= p$ satisfied by a point $p = [x:y:z] \in \mathbb{P}^2_\mathbb{H}$  for  infinitely many values of $n \in \Z$ then $p=[1: 0: 0]=e_1$ and hence $L_0(G)=\{e_1\}$. Note that for each $n \in \N$, $\frac{2}{n(n-1)}g^n$ and $\frac{2}{n(n+1)}g^{-n}$  are lifts of  $g^{n}$ and $g^{-n}$, respectively. Thus it follows that both the sequences $g^{n}$ and $g^{-n}$ converge to $D_5= \begin{pmatrix}
			0 & 0& 1 \\0 & 0 & 0 \\0 & 0 & 0
		\end{pmatrix}$, where $  Ker(D_5)=\mathbb{L}{ \{ e_1,e_2 \}} = \{ [x:y:z] \in  \mathbb{P}^2_\mathbb{H} \mid z =0    \}$. Thus  for each point $p \in \mathbb{P}^2_\mathbb{H}\setminus \mathbb{L} { \{e_1,e_2\}}$, sequences $g^{n}(p)$ and $g^{-n}(p)$ converge to $e_1$ as $n\rightarrow \infty$. Since $L_0(G)=\{e_1\}\subset \mathbb{L}\{e_1,e_2\}$, it follows that  $ e_1$ is the only cluster point of orbits $\{g^n(p)\}_{n \in \Z}$ for every point $ p \in \mathbb{P}^2_\mathbb{H}\setminus L_0(G)$. Hence,  $L_1(G)=\{e_1\}$.
		
		To find $L_2(G)$, consider a closed set $C=\mathbb{L}{\{e_1,e_2\}} \subset \mathbb{P}^2_\mathbb{H}$.
		Now note that $D_5$  induces a  pseudo-projective transformation in $\mathrm{QP}(3,\H)$ such that $\mathbb{L}{ \{ e_1,e_2 \}}= Ker(D_5)$. Therefore, using Lemma \ref{lem:app1}  we can show that the set of cluster points of orbits of compact subsets of  $\mathbb{P}^2_\mathbb{H}\setminus C$ lies in  $L_0(G)\cup L_1(G)=\{e_1\}$. Then, using Lemma \ref{lemma:L2} we have  that  $ L_2(G) \subset 
		\mathbb{L} {\{e_1,e_2\}}$. Now to show that $\mathbb{L} {\{e_1,e_2\}} \subset L_2(G)$, consider a compact subset $K = \{[x:-\frac{n-1}{n}:\frac{2}{n}] \in \mathbb{P}^2_\mathbb{H} \mid n\in \mathbb{N}\} \cup \{[x:-1:0]\}$ of $\mathbb{P}^2_\mathbb{H}\setminus \{e_1\}$, where $ \{e_1\} = L_0(G) \cup L_1(G)$. Then the sequence $g^n([x:-\frac{n-1}{n}:\frac{2}{n}]) = [x:1 + \frac{1}{n}:\frac{2}{n}]$ converges to a point  $[x:1 :0] \in \mathbb{L}{\{e_1,e_2\}}$ as $n\rightarrow \infty$. This implies $\mathbb{L}{\{e_1,e_2\}} \subset L_2(G)$. Hence,  $ L_2(G)=\mathbb{L}{\{e_1,e_2\}} $.
		
		\item 	\medskip  { \textbf{Ellipto-parabolic:}} In this case, for each $n \in \N$, we have 
		$$g^n \begin{bmatrix}
			x\\y\\z
		\end{bmatrix} = \left(\begin{array}{ccc}
			e^{2\pi i n\alpha} & ne^{2\pi i (n-1)\alpha} & 0\\0 & e^{2\pi i n\alpha} & 0\\ 0 & 0 & e^{2\pi i n\beta}
		\end{array}\right) \begin{bmatrix}
			x\\y\\z
		\end{bmatrix}=  \begin{bmatrix}
			e^{2\pi i n\alpha}x+ne^{2\pi i (n-1)\alpha}y\\ e^{2\pi i n\alpha}y\\ e^{2\pi i n\beta}z
		\end{bmatrix},  \hbox{and}$$ 
		$$g^{-n} \begin{bmatrix}
			x\\y\\z
		\end{bmatrix} = \left(\begin{array}{ccc}
			e^{-2\pi i n\alpha} &- ne^{-2\pi i (n+1)\alpha} & 0\\0 & e^{-2\pi i n\alpha} & 0\\ 0 & 0 & e^{-2\pi i n\beta}
		\end{array}\right) \begin{bmatrix}
			x\\y\\z
		\end{bmatrix}=  \begin{bmatrix}
			e^{-2\pi i n\alpha}x-ne^{-2\pi i (n+1)\alpha}y\\ e^{-2\pi i n\alpha}y\\ e^{-2\pi i n\beta}z
		\end{bmatrix}.$$
		
		Note that there exist subsequences of $\{\frac{1}{n}g^n(p)\}_{n\in \N}$ and  $\{\frac{-1}{n}g^{-n}(p)\}_{n\in \N}$  which converge to  $D_4= \begin{pmatrix}
			0 & 1& 0 \\0 & 0 & 0 \\0 & 0 & 0
		\end{pmatrix}$, where $  Ker(D_4)=\mathbb{L}{ \{ e_1,e_3 \}} $. 
		Thus for each point $p =[x:y:z] \in  \mathbb{P}^2_\mathbb{H}$ with $y \neq 0$, sequences $\{\frac{1}{n}g^n(p)\}_{n\in \N}$ and  $\{\frac{-1}{n}g^{-n}(p)\}_{n\in \N}$ 
		converge  to $e_1 \in  \mathbb{P}^2_\mathbb{H} $ as $n\rightarrow \infty$. Now to find $L_0(G)$ and $L_1(G)$, we can consider the following subcases depending on the rationality of $\alpha$ and $\beta$.
			\begin{enumerate}[(a)]
			\item  \medskip \textit{Rational ellipto-parabolic}: Here $\alpha$ and $\beta$ both are rational numbers. So there exists a $n_0 \in \N$ such that $e^{2\pi i n_0 \alpha} = e^{2\pi i n_0 \beta} =1$. Thus  a point $p= [x:y:z] \in \mathbb{P}^2_\mathbb{H}$  satisfy $g^n(p) =p$ for infinitely many values of $n \in Z$ only if $y =0$. Hence,  $L_0(G)=\mathbb{L}{ \{e_1,e_3\}}$. Futher, if $p =[x:y:z] \in  \mathbb{P}^2_\mathbb{H} \setminus \mathbb{L}{ \{e_1,e_3\}}$, then $y \neq 0$. Therefore,  $ e_1$ is the only cluster point of orbits $\{g^n(p)\}_{n \in \Z}$ for every point $ p \in \mathbb{P}^2_\mathbb{H}\setminus L_0(G)$. Hence,  $L_1(G)=\{e_1\}$. 
		
			\item 	   \textit{Irrational Ellipto-Parabolic}: In this case, either $\alpha$ or $\beta$ or both of them are irrational numbers.  The points which give infinite isotropy groups are only $e_1$ and $e_3$. Therefore,  $L_0(G)=\{e_1,e_3\}$.  Now note that for each point $ p =[x:y:z] \in \mathbb{P}^2_\mathbb{H}\setminus L_0(G)$ with $y \neq 0$, $ e_1$ is the only cluster point of orbits $\{g^n(p)\}_{n \in \Z}$. Also,  $g$ acts as an elliptic transformation on $\mathbb{L}{\{e_1,e_3\}}=\{[x:y:z] \mid y=0\}$  and hence $\mathbb{L}{\{e_1,e_3\}} \subset L_1(G)$ (cf. \textit{elliptic case}). Therefore, $ L_1(G) = \mathbb{L}{\{e_1,e_3\}}$.  
		\end{enumerate}
	
Therefore,  we have $L_0(G) \cup L_1(G) = \mathbb{L}{ \{e_1,e_3\}}$.
Now consider a compact set $K \subset  \mathbb{P}^2_\mathbb{H}\setminus L_0(G) \cup L_1(G)$. Using a similar argument as in the case of {\it vertical translation}, we can show that $e_1$ is the unique cluster point of the family of compact sets $\{g^n(K)\}_{n \in \Z}	$. Hence,  $ L_2(G) =\{e_1\}$.

\item 	\medskip { \textbf{Ellipto-translation:}}
In this case,  $g^n=\left(\begin{array}{ccc}
e^{2\pi i n \alpha} & ne^{2\pi i (n-1)\alpha} & \frac{n(n-1)}{2}e^{2\pi i(n-2)\alpha}\\
0 & e^{2\pi i n \alpha} & ne^{2\pi i (n-1)\alpha}\\
0 & 0 & e^{2\pi i n \alpha}
\end{array}\right)$.  Thus for each $n \in \N$, we have $$g^n\left[\begin{array}{c}
x \\y \\z
\end{array}\right]=\left[\begin{array}{c}
e^{2\pi i  n \alpha}x+ne^{2\pi i (n-1)\alpha}y+\frac{n(n-1)}{2}e^{2\pi i(n-2)\alpha}z\\
e^{2\pi i n \alpha}y+ne^{2\pi i (n-1)\alpha}z\\
e^{2\pi i n \alpha}z
\end{array}\right], \hbox{ and}$$ 
$$g^{-n}\left[\begin{array}{c}
x \\y \\z
\end{array}\right]=\left[\begin{array}{c}
e^{-2\pi i  n \alpha}x -ne^{-2\pi i (n+1)\alpha}y+\frac{n(n+1)}{2}e^{-2\pi i(n+2)\alpha}z\\
e^{-2\pi i n \alpha}y-ne^{-2\pi i (n+1)\alpha}z\\
e^{-2\pi i n \alpha}z
\end{array}\right].$$
Therefore, if the equation $g^n(p)= p$ satisfied by a point $p = [x:y:z] \in \mathbb{P}^2_\mathbb{H}$  for  infinitely many values of $n \in \Z$ then $p=[1: 0: 0]=e_1$ and hence $L_0(G)=\{e_1\}$. 
Note that there exist subsequences of $\{\frac{2}{n(n-1)}g^n(p)\}_{n\in \N}$ and  $\{\frac{2}{n(n+1)}g^{-n}(p)\}_{n\in \N}$  which converge to  $D_5= \begin{pmatrix}
0 & 0& 1 \\0 & 0 & 0 \\0 & 0 & 0
\end{pmatrix}$, where $  Ker(D_5)=\mathbb{L}{ \{ e_1,e_2 \}} $. Thus for each point $p =[x:y:z] \in  \mathbb{P}^2_\mathbb{H}$ with $z \neq 0$, both the sequences $\{\frac{2}{n(n-1)}g^n(p)\}_{n\in \N}$ and  $\{\frac{2}{n(n+1)}g^{-n}(p)\}_{n\in \N}$ 
converge  to $e_1 \in  \mathbb{P}^2_\mathbb{H} $ as $n\rightarrow \infty$.  It follows that $e_1$ is the only cluster point of  orbits $\{g^n(p)\}_{n \in \Z}$ for every point $ p \in \mathbb{P}^2_\mathbb{H}\setminus L_0(G)$ and so $L_1(G)=\{e_1\}.$

To find $L_2(G)$, consider a closed set $C=\mathbb{L}{\{e_1,e_2\}} \subset \mathbb{P}^2_\mathbb{H}$.
Now note that $D_5$  induces a  pseudo-projective transformation in $\mathrm{QP}(3,\H)$ such that $\mathbb{L}{ \{ e_1,e_2 \}}= Ker(D_5)$. Therefore, using Lemma \ref{lem:app1}  we can show that the set of cluster points of orbits of compact subsets of  $\mathbb{P}^2_\mathbb{H}\setminus C$ lies in  $L_0(G)\cup L_1(G)=\{e_1\}$. Then, using Lemma \ref{lemma:L2} we have $ L_2(G) \subset 
\mathbb{L} {\{e_1,e_2\}}$. 
Now to show that $\mathbb{L} {\{e_1,e_2\}} \subset L_2(G)$, we can suitably modify the example given in the case of {\it non-vertical translation}. 
Consider  a sequence $(k_n)_{n \in \N}$ such that $k_n = \left[	 x: \left(-1 + \frac{1}{n}\right) :\frac{2}{n} e^{2\pi i  \alpha}\right]$.  Note that sequence $(k_n)_{n \in \N} \subset K$, where $K = \{k_n \mid n\in \mathbb{N}\} \cup \{[x:-1:0]\}$ is a compact subset of $\mathbb{P}^2_\mathbb{H}\setminus L_0(G) \cup L_1(G)$.  Thus we have
$$g^n(k_n)=g^n\left(\left[	 x: \left(-1 + \frac{1}{n}\right) :\frac{2}{n} e^{2\pi i  \alpha}\right]\right) = \left[e^{2\pi i n\alpha}x:  \left(1 + \frac{1}{n}\right) e^{2\pi in \alpha} : \frac{2}{n}e^{2\pi i (n+1) \alpha}\right].$$ 
Then there exists a subsequence of $g^n(k_n)$ which  converges to a point  $[x:1 :0] \in \mathbb{L}{\{e_1,e_2\}}$ as $n\rightarrow \infty$. 
This implies $\mathbb{L}{\{e_1,e_2\}} \subset L_2(G)$. Hence,  $ L_2(G)=\mathbb{L}{\{e_1,e_2\}} $.
\end{enumerate}
This completes the proof.	\end{proof}

\section{Extended Conze-Guivarc'h Limit Set}\label{sec:conze}
In this section,  we take the natural action of $\mathrm{PSL}(3,\mathbb{H})$ on  the dual space $(\mathbb{P}^2_\mathbb{H})^*$. This action is given by   $g\cdot l\mapsto (g^{-1})^T\begin{bmatrix}
a\\b\\c
\end{bmatrix}$, where $g\in \mathrm{PSL}(3,\mathbb{H})$ and $l \in (\mathbb{P}^2_\mathbb{H})^*$ is the quaternionic projective line   with polar vector  $[a:b:c]\in \mathbb{P}^2_\mathbb{H}$.  
We define lines $l_1,l_2,l_3\in (\mathbb{P}^2_\mathbb{H})^*$ as  
 $l_i=\{[x_1:x_2:x_3]\in \mathbb{P}^2_\mathbb{H} \mid x_i=0 \}, \hbox{where} \,\, i \in \{1,2,3\}.$
 Note that for each $i$, ${l_i}$ is polar to the point ${e_i}$ of $\mathbb{P}^2_\mathbb{H}$.
Define the open sets $U_i$ as $U_i=\{[z_1:z_2:z_3] \in (\mathbb{P}^2_\mathbb{H})^* \mid  z_i\neq 0\}$, where $i=1,2,3$. 

Now recall the notion of the extended Conze-Guivarc'h limit set, see  \defref{def:conze2}.
In the next theorem,  we will prove Table \ref{table:4}  for  the extended Conze-Guivarc'h limit sets.
\begin{theorem}
	The extended Conze-Guivarc'h limit sets for the cyclic subgroups of  $\mathrm{PSL}(3,\mathbb{H})$ are the following:
	\begin{enumerate}[(i)]
		\item For the elliptic case, it is $\hat{L}(G)=\phi$ or $(\mathbb{P}^2_\mathbb{H})^*$ according to the fact that a lift of the generating element is of finite or infinite order.
		\item For the parabolic case, it is:
		\begin{enumerate}
			\item $\hat{L}(G)=\mathbb{L}\{e_1,e_3\}$ for vertical translation,
			\item $\hat{L}(G)=\mathbb{L}\{e_1,e_2\}$ for non-vertical translation,
			\item $\hat{L}(G)=\mathbb{L}\{e_1,e_3\}$ for ellipto-parabolic	and
			\item $\hat{L}(G)=\mathbb{L}\{e_1,e_2\}$ for ellipto-translation.
		\end{enumerate}
		\item For the loxodromic case, it is:
		\begin{enumerate}
			\item $\hat{L}(G)=\mathbb{L}\{e_1,e_2\}\cup\mathbb{L}\{e_2,e_3\}$ for regular loxodromic,
			\item $\hat{L}(G)=\mathbb{L}\{e_1,e_2\}$ for screw,
			\item $\hat{L}(G)=\mathbb{L}\{e_1,e_2\}$ for homothety and
			\item $\hat{L}(G)= \mathbb{L}\{e_1,e_2\} \cup \mathbb{L}\{e_1,e_3\}$ for loxo-parabolic.
		\end{enumerate}
	\end{enumerate}
\end{theorem}
\begin{proof}
	We will consider various cases depending on the conjugacy classes to determine the extended Conze-Guivarc'h limit sets.
	\begin{enumerate}[(i)]
		\item  \textbf{Elliptic element:} In this case, we have 
		$g=\left(\begin{array}{ccc}
				e^{2\pi i \alpha} & 0 & 0\\0 & e^{2\pi i \beta} & 0\\ 0 & 0 & e^{2\pi i \gamma}
			\end{array}\right),  \alpha,\beta,\gamma \in \mathbb{R}.$
	Thus an elliptic element $g$ is of finite order  only if $\alpha,\beta, \gamma\in \mathbb{Q}$. In this case
		$$g^n\cdot \left[\begin{array}{c}
			a\\b\\c
		\end{array}\right]=(g^{-n})^T
		\left[\begin{array}{c}
			a\\b\\c
		\end{array}\right]=\left(\begin{array}{ccc}
			e^{-2\pi n i \alpha} & 0 & 0\\0 & e^{-2\pi n i \beta} & 0\\ 0 & 0 & e^{-2\pi ni \gamma}
		\end{array}\right)\left[\begin{array}{c}
			a\\b\\c
		\end{array}\right]=\left[\begin{array}{c}
			e^{-2\pi n i \alpha}a\\e^{-2\pi n i \beta}b\\e^{-2\pi ni \gamma}c
		\end{array}\right].$$ 
		
		For finite order case $\hat{L}(G)=\phi$ and for infinite order case it is $(\mathbb{P}^2_\mathbb{H})^*$. To see this, for each  $q=[a:b:c]\in (\mathbb{P}^2_\mathbb{H})^*$, consider open set $U= (\mathbb{P}^2_\mathbb{H})^*$.

		\item \medskip \textbf{Parabolic element:}
		\begin{enumerate}[(a)]
			\item  \textbf{Vertical translation:} In this case,  $	g=\left(\begin{array}{ccc}
					1 & 1 & 0\\0 & 1 & 0\\ 0 & 0 & 1
				\end{array}\right).$
			This case is similar to the corresponding complex case (cf. \cite{DU}), therefore $\hat{L}(G)=\mathbb{L}{ \{e_1,e_3\}}\subset \mathbb{P}^2_\mathbb{H}$. To see this, consider $q=[0:1:0]\in (\mathbb{P}^2_\mathbb{H})^*$ and  open set $U=U_1$.
			
			\item \medskip \textbf{Non-vertical translation:} In this case,  $g=\left(\begin{array}{ccc}
					1 & 1 & 0\\0 & 1 & 1\\ 0 & 0 & 1
				\end{array}\right).$
			This case is similar to the corresponding complex case (cf. \cite{DU}), so $\hat{L}(G)=\mathbb{L}{ \{e_1,e_2\}}\subset \mathbb{P}^2_\mathbb{H}$. To see this, consider  $q= [0:0:1]\in (\mathbb{P}^2_\mathbb{H})^*$  and open set $U=U_1$.
			
			\item \medskip \textbf{Ellipto-parabolic:} In this case, $g=\left(\begin{array}{ccc}
				e^{2\pi i \alpha} & 1 & 0\\0 & e^{2\pi i \alpha} & 0\\ 0 & 0 & e^{2\pi i \beta}
			\end{array}\right).$ Then for each  line $p=[a:b:c]\in (\mathbb{P}^2_\mathbb{H})^*$, we have $g^n \cdot p = (g^{-n})^T \left[\begin{array}{c}
				a\\b\\c
			\end{array}\right]$
			$$
			=  \left(\begin{array}{ccc}
				e^{-2\pi n i n\alpha} & 0 & 0\\ne^{-2\pi i (n-1)\alpha} & e^{2\pi i n\alpha} & 0\\0 & 0 & e^{-2\pi n i n\beta}
			\end{array}\right)  \left[\begin{array}{c}
				a\\b\\c
			\end{array}\right]=\left[\begin{array}{c}
				e^{-2\pi n i n\alpha}a\\ne^{-2\pi i (n-1)\alpha}a+e^{2\pi i n\alpha}b\\e^{-2\pi n i n\beta}c
			\end{array}\right].$$
			
			Now consider the open set $U = U_1 \subset (\mathbb{P}^2_\mathbb{H})^*$. Note that sequence  $\frac{1}{n}(g^{-n})^T \left[\begin{array}{c}
				a\\b\\c
			\end{array}\right]$  converges to the line $q=[0:1:0]\in (\mathbb{P}^2_\mathbb{H})^*$ as $n\rightarrow \infty$ whenever $a\neq0$.  Therefore,  $\hat{L}(G)=\mathbb{L}\{e_1,e_3\}\subset \mathbb{P}^2_\mathbb{H}$.
			
			\item \medskip \textbf{Ellipto-translation:} In this case, $	g=\left(\begin{array}{ccc}
				e^{2\pi i \alpha} & 1 & 0\\0 & e^{2\pi i \alpha} & 1\\ 0 & 0 & e^{2\pi i \alpha}
			\end{array}\right)$. This implies
			$$ (g^{-n})^T = \left(\begin{array}{ccc}
				e^{-2\pi i n \alpha} & 0 & 0\\
				-ne^{-2\pi i (n+1)\alpha} & e^{-2\pi i n \alpha} & 0\\
				\frac{n(n+1)}{2}e^{-2\pi i(n+2)\alpha} & -ne^{-2\pi i (n+1)\alpha} & e^{-2\pi i n \alpha}
			\end{array}\right).$$
			Thus for each  line $p=[a:b:c]\in (\mathbb{P}^2_\mathbb{H})^*$, we have $$g^n \cdot p = (g^{-n})^T \left[\begin{array}{c}
				a\\b\\c
			\end{array}\right]
			= \left[\begin{array}{c}
				e^{-2\pi i n \alpha}a\\-ne^{-2\pi i (n+1)\alpha}a+e^{-2\pi i n \alpha}b\\\frac{n(n+1)}{2}e^{-2\pi i(n+2)\alpha}a-ne^{-2\pi i (n+1)\alpha}b+e^{-2\pi i n \alpha}c
			\end{array}\right].$$
			
			Now consider the open set $U = U_1 \subset (\mathbb{P}^2_\mathbb{H})^*$. Note that sequence $\frac{2}{n(n+1)}(g^{-n})^T \left[\begin{array}{c}
				a\\b\\c
			\end{array}\right]$  converges to the line $q=[0:0:1]\in (\mathbb{P}^2_\mathbb{H})^*$ as $n\rightarrow \infty$ whenever $a \neq0$.  Therefore,  $\hat{L}(G)=\mathbb{L}\{e_1,e_2\}\subset \mathbb{P}^2_\mathbb{H}$.
		\end{enumerate}
		
\item  \medskip \textbf{Loxodromic element:} 
		\begin{enumerate}[(a)]
			\item  \textbf{Regular loxodromic}: In this case,  
			$g=\left(\begin{array}{ccc}
				\lambda & 0 & 0\\0 & \mu & 0\\0 & 0 & \xi
			\end{array}\right),\; |\lambda|<|\mu|<|\xi|$.  
			Then $g^n\cdot p=(g^{-n})^T  \left[\begin{array}{c}
				a\\b\\c
			\end{array}\right]=\left(\begin{array}{ccc}
				\lambda^{-n}a & 0 & 0\\0 & \mu^{-n} b & 0\\0 & 0 & \xi^{-n}c
			\end{array}\right)$. 
			Now conisder the open set $U_1 =\{[a:b:c] \in (\mathbb{P}^2_\mathbb{H})^* \mid  a \neq 0\}$  and lift $|\lambda|^{n}(g^{-n})^T$ of $(g^{-n})^T$. Then the  sequence  $g^n \cdot p$ converges to the line $q =[1:0:0] \in (\mathbb{P}^2_\mathbb{H})^*$ as $n\rightarrow \infty$ for every line $p =[a:b:c] \in U_1$. Therefore,  $\mathbb{L}\{e_2,e_3\} \subset \hat{L}(G)$.
			Further, conisder the open set $U_3 =\{[a:b:c] \in (\mathbb{P}^2_\mathbb{H})^* \mid  c\neq 0\}$  and lift $\frac{1}{|\xi|^{n}}(g^{n})^T$ of $(g^{n})^T$. Then the  sequence  $g^{-n} \cdot p$ converges to the line $q =[0:0:1] \in (\mathbb{P}^2_\mathbb{H})^*$ as $n\rightarrow \infty$ for every line $p =[a:b:c] \in U_3$. Therefore,  $\mathbb{L}\{e_1,e_2\} \subset \hat{L}(G)$. Hence, we have  $\hat{L}(G)=\mathbb{L}\{e_1,e_2\} \cup \mathbb{L}\{e_2,e_3\}$.

			\item   \textbf{Screw loxodromic}: 
			In this case, $g=\left(\begin{array}{ccc}
				\lambda & 0 & 0\\0 & \mu & 0\\0 & 0 & \xi
			\end{array}\right),$ where $ |\lambda|=|\mu|\neq 1, \lambda \neq \mu$ and $ |\xi|=1/{|\lambda|^2}$. 
			Then $g^n\cdot p=(g^{-n})^T  \left[\begin{array}{c}
				a\\b\\c
			\end{array}\right]=\left(\begin{array}{ccc}
				\lambda^{-n}a & 0 & 0\\0 & \mu^{-n} b & 0\\0 & 0 & \xi^{-n}c
			\end{array}\right)$. 
			
			Conisder  the open set $U_3 =\{[a:b:c] \in (\mathbb{P}^2_\mathbb{H})^* \mid  c\neq 0\}$ 
			and lift $\frac{1}{|\lambda|^{2n}}(g^{-n})^T$ of $(g^{-n})^T$. Now it is not hard to check that $g^n \cdot p$ converges to the line $q =[0:0:1] \in (\mathbb{P}^2_\mathbb{H})^*$ as $n\rightarrow \infty$ for every line $p =[a:b:c] \in U_3$. Therefore,  $\hat{L}(G)=\mathbb{L}\{e_1,e_2\}$.
			
			\item   \textbf{Homothety}: In this case,  $g=\left(\begin{array}{ccc}
				\lambda & 0 & 0\\0 & \lambda & 0\\0 & 0 & \xi
			\end{array}\right), \; |\lambda|\neq 1,\; |\xi|=1/{|\lambda|^2}$. Then $(g^{-n})^T=\left(\begin{array}{ccc}
				\lambda^{-n} & 0 & 0\\0 & \lambda^{-n} & 0\\0 & 0 & \xi^{-n}
			\end{array}\right)$. Now we can show that $\hat{L}(G)=\mathbb{L}\{e_1,e_2\}$ by using a similar argument as in the case of {\it screw loxodromic}.
			
			\item  \medskip \textbf{Loxo-parabolic:} In this case, $g=\left(\begin{array}{ccc}
				\lambda & 1 & 0\\0 & \lambda & 0\\0 & 0 & \xi
			\end{array}\right), \; |\xi|=|\lambda|^{-2}\neq 1$. Then $$ g^n\cdot p= (g^{-n})^T\begin{bmatrix}
				a\\b\\c
			\end{bmatrix}=\left(\begin{array}{ccc}
				\lambda^{-n} & 0 & 0\\ -n\lambda^{-n-1} & \lambda^{-n} & 0\\0 & 0 & \xi^{-n}
			\end{array}\right)\begin{bmatrix}
				a\\b\\c
			\end{bmatrix}=\begin{bmatrix}\lambda^{-n}a\\-n\lambda^{-(n+1)}a+\lambda^{-n}b\\ |\lambda|^{2n} e^{-in\arg(\xi)} c\end{bmatrix}.$$
			Now conisder the open set $U_3 =\{[a:b:c] \in (\mathbb{P}^2_\mathbb{H})^* \mid  c\neq 0\}$  and lift $\frac{1}{|\lambda|^{2n}}(g^{-n})^T$ of $(g^{-n})^T$. Then the  sequence  $g^n \cdot p$ converges to the line $q =[0:0:1] \in (\mathbb{P}^2_\mathbb{H})^*$ as $n\rightarrow \infty$ for every line $p =[a:b:c] \in U_3$. Therefore,  $\mathbb{L}\{e_1,e_2\} \subset \hat{L}(G)$.
				Further, conisder the open set $U_1 =\{[a:b:c] \in (\mathbb{P}^2_\mathbb{H})^* \mid  a\neq 0\}$  and lift $\frac{1}{n|\lambda|^{n-1}}(g^{n})^T$ of $(g^{n})^T$. Then the  sequence  $g^{-n} \cdot p$ converges to the line $q =[0:1:0] \in (\mathbb{P}^2_\mathbb{H})^*$ as $n\rightarrow \infty$ for every line $p =[a:b:c] \in U_1$. Therefore,  $\mathbb{L}\{e_1,e_3\} \subset \hat{L}(G)$. Hence, we have  $\hat{L}(G)=\mathbb{L}\{e_1,e_2\} \cup \mathbb{L}\{e_1,e_3\}$.
		\end{enumerate}
\end{enumerate}
	This completes the proof. 
\end{proof}

\subsubsection*{Acknowledgement} 
We are thankful to Jos\'e Seade and Angel Cano for their comments and suggestions on this paper. The possibility of generalising this paper to arbitrary dimensions was kindly pointed out by Seade. We hope to return to that in a future work. 

Gongopadhyay is partially supported by the SERB core research grant CRG/2022/003680. Lohan acknowledges full support from the CSIR SRF grant, file No.: 09/947(0113)/
2019-EMR-I, during the course of this work. 
\subsubsection*{Data Availability.} Data sharing is not applicable to this article as no datasets were generated or analysed during
the current study.\\ \\
\textbf{Statements and Declarations.}\\
$\bullet$ \textbf{Conflicts of interest.} The authors have no relevant financial or non-financial interests to disclose. The authors do not have any conflict of interest.

\end{document}